\documentclass[12pt]{article}
\usepackage{amsmath,amsthm}
\usepackage{amsfonts}
\usepackage{latexsym}
\usepackage{indentfirst}
\usepackage{psfrag,epsf}
\usepackage{pstool}
\usepackage{appendix}
\usepackage{tikz}
\usepackage{graphicx, epsfig, subfigure}
\usepackage{epstopdf}
\usepackage{enumerate}
\usepackage[subnum]{cases}
\usepackage{geometry}
\usepackage{hyperref}
\usepackage{fullpage}
\usepackage{lineno}
\usetikzlibrary{shapes,arrows,calc}

\newtheorem{theorem}{Theorem}[section]\rm
\newtheorem{Claim}[theorem]{Claim}

\newtheorem{conjecture}[theorem]{Conjecture}
\newtheorem{construction}[theorem]{Construction}
\newtheorem{claim}[theorem]{Claim}

\newtheorem{definition}[theorem]{Definition}
\newtheorem{problem}[theorem]{Problem}

\newtheorem{remark}[theorem]{Remark}
\newtheorem{question}[theorem]{Question}
\theoremstyle{plain}

\DeclareMathOperator{\mad}{\mathrm{mad}}
\newtheorem*{THMMAIN}{Theorem~\ref{thm:modify}}
\newtheorem{Counting}[theorem]{Counting}

\newtheorem{subclaim}{Claim}[theorem] 

\newcommand{\tv}{ \textnormal{tokens}}
\newcommand{\pr}{\textnormal {primary}}
\newcommand{\bi}{\textnormal{\textsc{Big}}}
\newcommand{\ba}{\textnormal{\textsc{Basic}}}
\newcommand{\no}{\textnormal{\textsc{Nonbasic}}}

\def\tok{\textrm{tokens}}

\newcommand\bolder[1]{\contour{black}{#1}}

\usepackage[outline]{contour}%

\title{\bf Tight upper bound on the clique size  in the square of 2-degenerate graphs}
\author{Seog-Jin Kim\thanks{Department of Mathematics Education, Konkuk University,
Korea.  E-mail: {\tt skim12@konkuk.ac.kr
}},~Xiaopan Lian\thanks{Center for Combinatorics and LPMC, Nankai University, China.
 \emph{E-mail:}
{\tt lianxiaopanlxp@163.com} }
}

\begin{document}

\maketitle

\begin{abstract}
The {\em square} of a graph $G$, denoted $G^2$, has the
same vertex set as $G$ and has an edge between two vertices if the distance between them in $G$ is at most $2$. In general, $\Delta(G) + 1 \leq  \chi(G^2) \leq \Delta(G)^2 +1$ for every graph $G$.
Charpentier~\cite{cc}  asked whether $\chi(G^2) \leq 2 \Delta(G)$ if $mad(G) < 4$.  But Hocquard, Kim, and Pierron \cite{HKP} answered his question negatively.
For every even value of $\Delta(G)$, they constructed a 2-degenerate graph $G$ such that $\omega(G^2) = \frac{5}{2} \Delta(G)$. Note that if $G$ is a  2-degenerate graph, then $mad(G) < 4$. Thus, we have that
\[
{\displaystyle \frac{5}{2} \Delta(G) \leq \max \{\chi(G^2) : G \mbox{ is  a 2-degenerate graph} \} \leq 3 \Delta(G) +1}.
\]
So, it was naturally asked whether there exists a constant $D_0$ such that
$\chi(G^2) \leq \frac{5}{2} \Delta(G)$ if $G$ is a 2-degenerate graph with $\Delta(G) \geq D_0$.
Recently Cranston and Yu \cite{CY} showed that $\omega(G^2) \leq \frac{5}{2} \Delta(G)+72$ if $G$ is a 2-degenerate graph, and   $\omega(G^2) \leq \frac{5}{2} \Delta(G)+60$ if $G$ is a 2-degenerate graph with $\Delta(G) \geq 1729$. We show that there exists a constant $D_0$ such that $\omega(G^2) \leq \frac{5}{2} \Delta(G)$ if $G$ is a 2-degenerate graph
with $\Delta(G) \geq D_0$. This upper bound on $\omega(G^2)$ is tight by the construction in \cite{HKP}.
\end{abstract}


\section{Introduction}

The {\em square} of a graph $G$, denoted $G^2$, has the
same vertex set as $G$ and has an edge between two vertices if the distance between
them in $G$ is at most $2$.
Let $\chi(G)$ be the  chromatic number of a graph $G$, and let $\omega(G)$ be the clique number of $G$.
For every graph $G$, trivially $\Delta(G) + 1 \leq \chi(G^2) \leq \Delta(G)^2 + 1$.

\medskip
Wegner \cite{Wegner} proposed the following conjecture.

\begin{conjecture}\cite{Wegner} \label{conj-Wegner}
Let $G$ be a planar graph. The chromatic number
$\chi(G^2)$ of $G^2$ is at most 7 if $\Delta(G) = 3$,
at most $\Delta(G)+5$ if $4 \leq \Delta(G) \leq 7$, and at most $\lfloor \frac{3 \Delta(G)}{2} \rfloor +1$ if $\Delta(G) \geq 8$.
\end{conjecture}

Conjecture \ref{conj-Wegner} is wide open.  The only case for which we know tight bound is when $\Delta(G) = 3$. Thomassen \cite{Thomassen} and, independently,
Hartke, Jahanbekam, and Thomas \cite{Hartke}
showed that $\chi(G^2) \leq 7$ if $G$ is a planar graph with $\Delta(G)  = 3$, which proves  Conjecture \ref{conj-Wegner} for $\Delta(G)  = 3$.
Even  the upper bound on $\omega(G^2)$ was not known until this year.  Cranston \cite{Cranston23} showed that $\omega(G^2) \leq \lfloor \frac{3 \Delta(G)}{2} \rfloor +1$ if $G$ is a planar graph with  $\Delta(G) \geq 34$.  One may see the detailed story on the study of Wegner's conjecture in \cite{Cranston22}.

\medskip

The maximum average degree  of a graph $G$, denoted $\mad(G)$, is defined as $\max_{H \subseteq G} \frac{2|E(H)|}{|V(H)|}$.  A graph $G$ is $k$-degenerate if each subgraph of $G$ contains a
vertex of degree at most $k$. Equivalently, there exists a vertex order $\sigma$ such that each vertex has at most $k$ neighbors later in $\sigma$.
Note that if $G$ is a 2-degenerate graph, then $\mad(G) < 4$.

\medskip
Recently, following problem was considered in \cite{cc} and has received some attentions.
\begin{problem}[\cite{cc}] \label{Q-main}
For each integer $k \geq 2$, what is  $\max \{\chi(G^2) : \mad(G) < 2k\}$?
\end{problem}

For $k = 2$, Charpentier~\cite{cc} conjectured that
$\chi(G^2)\leq 2\Delta(G)$ if $\mad(G) < 4$, but it was disproved
in~\cite{KP} by constructing a graph $G$ such that $\mad(G)<4$ and
$\chi(G^2) = 2\Delta(G)+2$.

Furthermore, Hocquard, Kim, and Pierron \cite{HKP} constructed a graph $G$ with $\mad(G) < 4$ and $\omega(G^2) = \frac{5}{2} \Delta(G)$ for every even value of $\Delta(G)$.
They also showed that for every $G$ such that $\mad(G)<4$, the graph  $G^2$ is $3\Delta(G)$-degenerate.  So in general we have the following bounds on $\chi(G^2)$.

\begin{equation*}
\frac{5}{2} \Delta(G) \leq \max \{\chi(G^2) : \mad(G) < 4\} \leq 3 \Delta(G) + 1.
\end{equation*}
So the following question was asked in \cite{HKP}.
\begin{question} \label{question-HKP} \cite{HKP}
 Does there exist $D_0$ such that every graph with $\Delta(G) \geq D_0$ and $mad(G) < 4$ has
$\chi(G^2) \leq \frac{5}{2} \Delta(G)$?
\end{question}

Note that the construction in \cite{HKP} is actually a 2-degenerate graph.  So, the following question was also asked in \cite{HKP}.
\begin{question} \label{question-HKP2} \cite{HKP}
Does there exist $D_0$ such that every 2-degenerate graph $G$ with $\Delta(G) \geq D_0$  has
$\chi(G^2) \leq \frac{5}{2} \Delta(G)$?
\end{question}

In this direction, Cranston and Yu \cite{CY} showed the following interesting theorem.

\begin{theorem} \label{CY-thm}
Fix a positive integer $D$. If a graph G is 2-degenerate with $\Delta(G) \leq D$, then
$\omega(G^2) \leq \frac{5}{2} D + 72$.   Furthermore,
if $G$ is a 2-degenerate graph with $D \geq 1729$, then
$\omega(G^2) \leq \frac{5}{2} D + 60$.
\end{theorem}
Also it was proved in \cite{CY} that  if $\mad(G) < 4$ and $\Delta(G) \leq D$, then
$\omega(G^2) \leq \frac{5}{2} D + 532$.

\medskip
In this paper, we prove the following result.

\begin{theorem}\label{thm:Main}
There exists a constant $D_0>0$ such that for a positive integer $D$ with $D \geq D_0$, if $G$ is a 2-degenerate graph with $\Delta(G) \leq D$, then $\omega(G^2)\le \frac{5}{2}D$.
\end{theorem}

Note that $D_0 = 6\times( 331\times 2+10\times{331\choose 2}+2000)
=3,292,872$ in Theorem \ref{thm:Main}.
And
note that the upper bound on $\omega(G^2)$ in Theorem \ref{thm:Main} is tight by the construction in \cite{HKP}.

\medskip
This paper is organized as follows.  In Section \ref{Prelim}, we explain important results that were proved in \cite{CY}.  In Section \ref{section-proof}, we give the proof of  Theorem \ref{thm:Main}.

\section{Preliminaries} \label{Prelim}
In this section, first we provide the construction of $\omega(G^2) = \frac{5}{2} \Delta(G)$.

\begin{construction}\cite{HKP}
Fix a positive integer $D=4k+r$ where $k,r$ are positive integers with $k\ge 2$ and $0\le r\le 3$. We form graph $G_D$ as follows. Let $i$ and $j$ be positive integers with $1\le i<j\le 5$.
\begin{enumerate}[(1)]
\item  Start with the complete graph $K_5$ on vertex set $\{u_1,\ldots,u_5\}$.
\item  If $r=0$, then replace each edge $u_iu_j$ of $K_5$ with a copy of $K_{2,k}$, identifying the vertices in the part of size 2 with $u_i$ and $u_j$ for each $1\le i<j\le 5$;\\

If $r=1$, then replace $u_iu_j\in \{u_1u_2,u_3u_4\} $ with a copy of $K_{2,k+1}$ and  replace each edge $u_iu_j\in E(K_5)\setminus \{u_1u_2,u_3u_4\}$ with a copy of $K_{2,k}$, identifying the vertices in the part of size 2 with $u_i$ and $u_j$ for each $1\le i<j\le 5$;\\

If $r=2$, then replace each edge $u_iu_j\in \{u_1u_2,u_2u_3,u_3u_4,u_4u_5,u_5u_1\}$ with a copy of $K_{2,k+1}$ and replace each edge $u_iu_j\in E(K_5)\setminus\{u_1u_2,u_2u_3,u_3u_4,u_4u_5,u_5u_1\}$ with a copy of $K_{2,k}$, identifying the vertices in the part of size 2 with $u_i$ and $u_j$ for each $1\le i<j\le 5$;\\

If $r=3$, then replace each edge $u_iu_j\in \{u_1u_2,u_3u_4\}$ with a copy of $K_{2,k+2}$,  replace each edge $u_iu_j\in \{u_2u_3,u_4u_5,u_5u_1\}$ with a copy of $K_{2,k+1}$, and replace each edge $u_iu_j\in E(K_5)\setminus\{u_1u_2,u_2u_3,u_3u_4,u_4u_5,u_5u_1\}$ with a copy of $K_{2,k}$, identifying the vertices in the part of size 2 with $u_i$ and $u_j$ for each $1\le i<j\le 5$;\\

We denote by $V_e$ the $k$, $k+1$ or $k+2$ vertices added while replacing the edge $e$ of $K_5$.

\item  Now for each pair, $x$ and $y$, of vertices of degree 2 with no common neighbor, add a vertex $z_{xy}$ adjacent to both $x$ and $y$.
\end{enumerate}
Call the resulting graph $G_D$ (see Figure~\ref{fig:tight}).

Observe that $G_D$ is 2-degenerate with $\omega(G^2_D)=\lfloor\frac{5}{2}D\rfloor$ and $\Delta(G_D)=D$.  Note that when $r\in \{0,2\}$, this is the degree of vertices of the original $K_5$ and every other vertex has degree no more than $D$.  When $r\in \{1,3\}$, this is the degree of all vertices but $u_5$ of the original $K_5$ and every other vertex has degree no more than $D$.) Moreover, the vertices in $\bigcup_{e\in E(K_5)}V_e$
induce a clique of size $\lfloor\frac{5}{2}D\rfloor$ in $G^2_D$.
\end{construction}

\begin{figure}[!ht]
\centering
\begin{tikzpicture}[v/.style={fill=gray,minimum size =10pt,ellipse,inner sep=1pt},node distance=1.5cm,v2/.style={fill=black,minimum size =5pt,ellipse,inner sep=1pt},node distance=1.5cm,scale=0.75]
\node[v] (v1) at (90:4){};
\node[v] (v2) at (162:4){};
\node[v] (v3) at (234:4){};
\node[v] (v4) at (306:4){};
\node[v] (v5) at (18:4){};
\node[v2] (vv11) at (126:3.3){};
\node[v2] (vv21) at (198:3.3){};
\node[v2] (vv31) at (270:3.3){};
\node[v2] (vv41) at (342:3.3){};
\node[v2] (vv51) at (54:3.3){};

\draw[bend right=50,gray] (vv11) to node[v2,midway,gray] {} (vv31);
\draw[bend left=50,gray] (vv11) to node[v2,midway,gray] {} (vv41);
\draw[bend left=50,gray] (vv21) to node[v2,midway,gray] {} (vv51);
\draw[bend right=50,gray] (vv21) to node[v2,midway,gray] {} (vv41);
\draw[bend right=50,gray] (vv31) to node[v2,midway,gray] {} (vv51);

\node[v2] (vv12) at (126:3.4){};
\node[v2] (vv22) at (198:3.4){};
\node[v2] (vv32) at (270:3.4){};
\node[v2] (vv42) at (342:3.4){};
\node[v2] (vv52) at (54:3.4){};

\node[v2] (vv13b) at (162:1.3){};
\node[v2] (vv35b) at (306:1.3){};
\node[v2] (vv52b) at (90:1.3){};
\node[v2] (vv24b) at (234:1.3){};
\node[v2] (vv41b) at (18:1.3){};

\draw[bend left=15,gray] (vv11) to node[v2,midway,gray] {} (vv35b);
\draw[bend left=15,gray] (vv21) to node[v2,midway,gray] {} (vv41b);
\draw[bend left=15,gray] (vv31) to node[v2,midway,gray] {} (vv52b);
\draw[bend left=15,gray] (vv41) to node[v2,midway,gray] {} (vv13b);
\draw[bend left=15,gray] (vv51) to node[v2,midway,gray] {} (vv24b);

\draw[bend right=80,gray,distance=1.2cm] (vv13b) to node[v2,midway,gray] {} (vv24b);
\draw[bend right=80,gray,distance=1.2cm] (vv24b) to node[v2,midway,gray] {} (vv35b);
\draw[bend right=80,gray,distance=1.2cm] (vv35b) to node[v2,midway,gray] {} (vv41b);
\draw[bend right=80,gray,distance=1.2cm] (vv41b) to node[v2,midway,gray] {} (vv52b);
\draw[bend right=80,gray,distance=1.2cm] (vv52b) to node[v2,midway,gray] {} (vv13b);

\node[v2] (vv1) at (126:3.2){};
\node[v2] (vv2) at (198:3.2){};
\node[v2] (vv3) at (270:3.2){};
\node[v2] (vv4) at (342:3.2){};
\node[v2] (vv5) at (54:3.2){};

\node[v2] (vv13c) at (162:1.4){};
\node[v2] (vv35c) at (306:1.4){};
\node[v2] (vv52c) at (90:1.4){};
\node[v2] (vv24c) at (234:1.4){};
\node[v2] (vv41c) at (18:1.4){};

\node[v2] (vv13a) at (162:1.2){};
\node[v2] (vv35a) at (306:1.2){};
\node[v2] (vv52a) at (90:1.2){};
\node[v2] (vv24a) at (234:1.2){};
\node[v2] (vv41a) at (18:1.2){};

\draw (v1) -- (vv1) -- (v2) -- (vv2) -- (v3) -- (vv3) -- (v4) -- (vv4) -- (v5) -- (vv5) -- (v1);
\draw (v1) -- (vv11) -- (v2) -- (vv21) -- (v3) -- (vv31) -- (v4) -- (vv41) -- (v5) -- (vv51) -- (v1);
\draw (v1) -- (vv12) -- (v2) -- (vv22) -- (v3) -- (vv32) -- (v4) -- (vv42) -- (v5) -- (vv52) -- (v1);
\draw (v1) -- (vv13a) -- (v3) -- (vv35a) -- (v5) -- (vv52a) -- (v2) -- (vv24a) -- (v4) -- (vv41a) -- (v1);
\draw (v1) -- (vv13b) -- (v3) -- (vv35b) -- (v5) -- (vv52b) -- (v2) -- (vv24b) -- (v4) -- (vv41b) -- (v1);
\draw (v1) -- (vv13c) -- (v3) -- (vv35c) -- (v5) -- (vv52c) -- (v2) -- (vv24c) -- (v4) -- (vv41c) -- (v1);

\end{tikzpicture}
\caption{Black vertices induce a clique in $G^2_D$ }\label{fig:tight}
\end{figure}
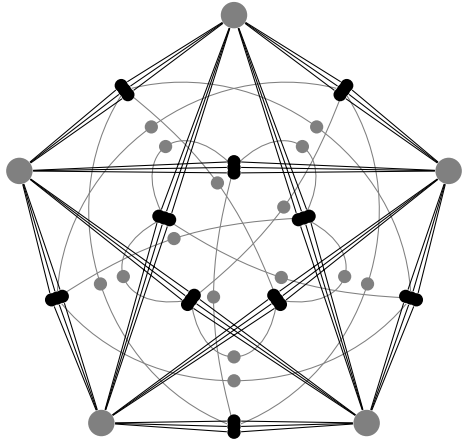

Next, we introduce a definition used in \cite{CY}.

\begin{definition} \label{nice} \rm
A graph $G$ is nice w.r.t.~a clique $S$ in $G^2$, if the following (1), (2), and (3) hold.
\begin{enumerate}[(1)]
\item $S$ is a clique in $G^2$,

\item $S$ is an independent set in $G$, and

\item  $G$ has a 2-degeneracy order  $\sigma$ such that all vertices of $S$ appear consecutively in $\sigma$.
\end{enumerate}
\end{definition}

\medskip
Note that the notion in Definition \ref{nice} is inspired by the construction
 in \cite{HKP} (Figure \ref{fig:tight}).
And the following two theorems were proved in \cite{CY}.

\begin{theorem}\label{thm:cynice}\cite{CY}
If $G$ is nice w.r.t.~a maximum clique $S$ in $G^2$ and $\Delta(G)\le D$, then $\omega(G^2)\le \frac{5}{2}D$.
\end{theorem}

Fix a positive integer $D$. Let $f(D)$ be the maximum size of a clique $S$ in the square of a
2-degenerate graph $G$ with $\Delta(G)\le D$. Then we have the following important property.

\begin{theorem}\label{thm:cymainthm}(\cite{CY},Theorem 5)
For every positive integer $D \geq 1729$,
some 2-degenerate
graph $G$ with $\Delta(G)\le D$ is nice w.r.t.~a clique $S$ in $G^2$ and $|S|\ge f(D)-60$.
\end{theorem}

Note that Theorem \ref{CY-thm} comes from Theorem \ref{thm:cynice} and Theorem \ref{thm:cymainthm}.
We also use Theorem \ref{thm:cymainthm} in our proof of the main theorem.  But, in Theorem~\ref{thm:cymainthm}, the authors of \cite{CY} assumed that $f(D)\ge \frac{5}{2}D+60$ which is slightly different from our assumption that $f(D)> \frac{5}{2}D$. So, we modify the values in calculations in the proof of Theorem \ref{thm:modify} and we have the following theorem.

\begin{theorem}\label{thm:modify}
For every positive integer $D \geq 1733$,
if a 2-degenerate
graph $G$ with $\Delta(G)\le D$ satisfies $\omega(G^2)=f(D)>\frac{5}{2}D$, then there is a subgraph $G^*$ of $G$ which is nice w.r.t.~a clique $S^*$ in ${(G^*)}^2$ and $|S^*|\ge f(D)-60$.
\end{theorem}

For completeness of the proof, we provide the proof of Theorem \ref{thm:modify} in the appendix.  However, we shall emphasize that the proof of Theorem \ref{thm:modify} is exactly that of Theorem 5 in \cite{CY}.  
The proof of Theorem~\ref{thm:modify} just follows the proof of 
Theorem 5 in \cite{CY}.


\section{Proof of Theorem~\ref{thm:Main}} \label{section-proof}
Let $D_0:=6\times(331\times 2+10\times{331\choose 2}+2000)$.
Suppose that for an integer $D \ge D_0$,
there exists a 2-degenerate graph $G$ with $\Delta(G) \leq D$ such that
$G^2$ has a clique $S$ of order $f(D)$ where $f(D)>\frac{5}{2}D$.
Throughout  this section,  let $G$ be a counterexample to Theorem \ref{thm:Main} minimizing $|V (G)|+|E(G)|$.

Let $\sigma$ be a vertex order witnessing that $G$ is 2-degenerate. Choose $\sigma$ so that the first vertex in $S$ appears as late as possible (each vertex has at most two neighbors later in the order $\sigma$).  Since $G$ is a minimal counterexample, every edge in $E(G)$ has an endpoint in $S$. Hence, by the proof of Theorem~\ref{thm:modify},  there is  a triple $(G^*,S^*,\sigma^*)$
obtained from $(G,S,\sigma)$ such that
 \begin{enumerate}[(a)]
 \item $S^*\subseteq S$ and  $|S\setminus S^*|\le 60$,
 \item $G^*$ is obtained from $G$ by deleting every edge of $G$  whose endpoints are both outside $S^*$,
 \item $S^*$ is a clique in $(G^*)^2$,
 \item $G^*$ is nice w.r.t. $S^*$ where the order $\sigma^*$ satisfies that all vertices of $S^*$ appear consecutively in $\sigma^*$, and
\item both of $V(G^*)\setminus S^*$ and $S^*$ are  independent sets in $G^*$.
 \end{enumerate}

Let $s^*:=|S^*|$.
\begin{remark}\label{sizeofS*}
Note that by Theorem~\ref{thm:cynice}, we have that $s^*\le \frac{5}{2}D$.  By (a) and our assumption that $|S|>\frac{5}{2}D$, we have that $s^*> \frac{5}{2}D-60$. Thus, $\frac{5}{2}D-60< s^*\le\frac{5}{2}D$.
\end{remark}
Note that throughout the paper, the bounds on $s^*$  are $\frac{5}{2}D-60< s^*\le\frac{5}{2}D$.

\medskip
Next, we define $R^*$ and $T^*$  as follows.
\begin{definition}
Let $T^*\subset V(G^*)$ be  the set of vertices which are vertices after $S^*$ in $\sigma^*$ and have at least one neighbor in $S^*$ in the graph $G^*$. Let $R^*=V(G^*)\setminus (S^*\cup T^*)$.
\end{definition}
\begin{remark}\label{rem:ScapV}
Observe that for each $v\in V(G)$, we have $N_G(v)\cap S^*=N_{G^*}(v)\cap S^*$. Furthermore, for each $v\in S^*$, we have that $N_{G^*}(v)=N_{G}(v)$. And since $G^*$ is 2-degenerate, by the definition of $R^*$, for each $v\in R^*$, we have that $|N_G(v)\cap S^*|=|N_{G^*}(v)\cap S^*|\le 2$.
\end{remark}

\medskip
\noindent {\bf [Outline of the proof Theorem \ref{thm:Main}]}
\begin{itemize}
\item If $f(D) > \frac{5}{2}D$, then we have that $|S \setminus S^*| \leq 60$ by Theorem~\ref{thm:modify}.  Also, by Theorem~\ref{thm:cynice}, we have that $|S^*|\le \frac{5}{2}D$.  So, our assumption is that  $\frac{5}{2}D -60<|S^*|\le \frac{5}{2}D$.
\item We establish an observation regarding the size of $T^*$ and show that $|T^*|  = 6$.
For this, we define an auxiliary multigraph, denoted  $H^*$, with $|S^*| = |E(H^*)|$ and $T^* = V(H^*)$.

\item If $|T^*| \geq 7$, then we show that $H^*$ has a very small degree vertex, which leads a contradiction.  This implies that $|T^*| \leq 6$.  Next, we show that
$|T^*| \geq 5$ by simple counting argument.  Finally, we show that if $|T^*| = 5$, then the graph $G$ has a vertex $v$ with $d_G(v) > D$, which is a contradiction.  Thus we conclude that
$|T^*| = 6 $.

\item Let $H_0^*$ be the underlying simple graph of $H^*$.
In Subsectin \ref{subsection-two}, we show that $H_0^*$ is a complete graph when $|T^*| = 6$.

\item From the fact that $H_0^*$ is a complete graph, we show that the multigraph $H^*$ is almost balanced. That is, each vertex $v \in T^*$ has $\frac{5}{6}D - c_1 \leq d_{H^*}(v) \leq \frac{5}{6}D + c_1'$, and  the multiplicity of each edge $uv$ in $H^*$ satisfies that $\mu(uv) \geq \frac{D}{6}-c_2$ where $c_1$, $c_1'$ and $c_2$ are small enough
compared to $D$.

\item For each $1 \leq i < j \leq 6$, let $S_{i,j}$ be the set of vertices in $S \setminus (S^*\cup \{v_i,v_j\})$ that are adjacent to neither $v_i$ nor $v_j$ where $v_i, v_j \in T^*$, and let $s_{i,j} = |S_{i, j}|$.
Next, from careful analysis, we obtain inequality~(\ref{counting}) after proving Claim \ref{size of S(ij)}.
$$\sum_{1\le i<j\le 6}s_{i,j}  \leq \left(\frac{5}{2}D-s^*\right)-30.$$
And then, we finish the proof of Theorem \ref{thm:Main} with careful analysis.
\end{itemize}

\subsection{Basic framework} \label{subsection-one}
In this subsection, we study the properties in $T^*$ and show that $|T^*| = 6$.
First, we show that  $T^*\neq \emptyset$ in Claim~\ref{claim:Tnotempty}. Then, we give a rough bound on the degree of vertices in the graph $G^*$ in Claim~\ref{claim:degvinT}.  Note that  the information given by Claim~\ref{claim:Tnotempty}  and Claim~\ref{claim:degvinT} is not sufficient to conclude that $|T^*|=6$. Therefore, we introduce an important
auxiliary multigraph, denoted  $H^*$, obtained from $G^*$ which satisfies that $V(H^*)=T^*$ and $|E(H^*)|=s^*$. Next, in Claim~\ref{lem:degreeinH*}, we investigate some properties of the multigraph $H^*$. Finally, using these  properties of the multigraph $H^*$, we show in Claim~\ref{size 6} that $ |T^*| = 6$.

\begin{claim}\label{claim:Tnotempty}
$T^*\neq \emptyset$.
\end{claim}
\begin{proof}
Suppose that $T^*=\emptyset$. Recall that $S^*$ is an independent set in $G^*$, so each edge in $(G^*)^2[S^*]$ comes via a vertex in $R^*$.  But each vertex $v \in S^*$ has at most $D$ neighbors $w$ in $R^*$ and and each such $w$ creates an adjacency in $(G^*)^2$ between $v$ and at most one vertex $v'$ in $S^*$, since $w$ has at most tow neighbors in $S^*$.
Then by $(d)$, we have that $|S^*|\le D+1$.  This is a contradiction since $s^*>\frac{5}{2}D-60>D+1$.
\end{proof}

By analyzing the adjacencies between vertices in $T^*$ and $S^*$ in the graph $G^*$, we prove that for each vertex $v\in T^*$, the size of $N_{G^*}(v)\cap S^*$ is somehow large.  It helps us to prove that the order of $T^*$ is bounded by combining Theorem~\ref{thm:cynice} and the handshaking lemma.   We now prove the following claim.

\begin{claim}\label{claim:degvinT}
\begin{itemize}
\item [(1)] For each $v\in S^*$, $|N_{G^*}(v)\cap T^*|=2$.
\item [(2)] For each $v\in T^*$, $|N_{G^*}(v)\cap S^*|\ge \frac{D}{2}-57$.
\end{itemize}
\end{claim}

\begin{proof}
(1) Suppose that there is a vertex $v\in S^*$ such that $|N_{G^*}(v)\cap T^*|\le 1$.
First note that $\Delta(G^*)\le D$.
Then since $S^*$ is an independent set, $v$ has at most $D - 1$ neighbors in $R^*$ and each of these neighbors has at most one neighbor in $S^*$ other than $v$ since $G^*$ is 2-degenerate.
And if $w \in N_{G^*}(v)\cap T^*$, then $w$ has at most $D-1$ neighbors distinct from $v$ in $S^*$.  Thus,
by (c), we have that
\begin{eqnarray*}
s^* & = & |N_{{(G^*)}^2}(v)\cap S^*|+1 \leq |N_{G^*}(v)\cap R^*| + |N_{G^*}(w)\cap S^*| + 1 \\
& \le &   D-1 + D -1 + 1
 = 2D - 1 <\frac{5}{2}D-60,
 \end{eqnarray*}
a contradiction.

\medskip
(2) Suppose that there is a vertex $v\in T^*$ such that $|N_{G^*}(v)\cap S^*|<\frac{D}{2}-57$. Let $u$ be a vertex in  $N_{G^*}(v)\cap S^*$.   By (1), $u$ has at most one neighbor, say $w$, distinct from $v$ in $T^*$.  Then
by (c) and $\Delta(G^*)\le D$,
\begin{eqnarray*}
s^* & = & |N_{{(G^*)}^2}(u)\cap S^*|+1 \le |N_{G^*}(u)\cap R^*| + |N_{G^*}(v)\cap S^*| +
|N_{G^*}(w)\cap S^*| \\
& < &
D-2+\frac{D}{2}-57+D-1=\frac{5}{2}D-60,
\end{eqnarray*}
a contradiction.
\end{proof}

\medskip
In the following, we show that the size of  $T^*$ is bounded. We prove this by constructing an auxiliary mulitigraph $H^*$ obtained from $G^*$.

\begin{construction} \label{construction-H*} \rm
Form $H^*$ from $G^*$ as follows: delete all vertices of $R^*$ and contract one edge incident to each vertex of $S^*$.
\end{construction}

Note that, by Claim~\ref{claim:degvinT} and Construction~\ref{construction-H*}, the following holds.
\begin{center}
$V(H^*)=T^*$, $|E(H^*)|=|S^*|=s^*$, and $\Delta(H^*)\le D$.
\end{center}

We define $H^*_0$, $\bar{H^*_0}$, $E(v)$, and $\mu(uv)$ as follows.
\begin{definition}
Denote by $H^*_0$ the underlying simple graph of $H^*$ and denote by $\bar{H^*_0}$ the complement of $H^*_0$. For $v\in T^*$, let $E(v)$ be the set of edges in $H^*$ that is incident with $v$. For an edge $uv\in E(H^*)$, let $\mu(uv)$ be the miltiplicity of $uv$ in $H^*$.
\end{definition}

First, we show the following key observation.

\begin{remark} \label{remark-basic}
Since $\Delta(G^*) \le D$, for each $v\in S^*$, $v$ has at most $D-2$ neighbors in $R^*$ by Claim \ref{claim:degvinT} (1). Thus, the edge in $H^*$ arising from $v$ has a common endpoint with all edges of $H^*$ except for at most $D-2$ edges.
In other words, for every edge $uv$ in $H^*$, we have that
$|E(H^*)\setminus(E(u) \cup E(v))| \leq D-2$.
\end{remark}

\medskip
We proceed by analyzing the adjacency between vertices in $H^*$. First, we establish the following claim.

\begin{claim}\label{lem:degreeinH*}
The following holds.
\begin{enumerate}[(1)]
\item  For each $uv\in E(H^*)$, $d_{H^*}(u)+d_{H^*}(v)-\mu(uv)\ge s^*-D+2$.

\item For each $v\in T^*$, $|N_{H^*}(v)|\ge 2$.

\item For each $v\in T^*$ and $w\in N_{H^*}(v)$,
${\displaystyle \sum_{u\in N_{H^*}(v)\setminus\{w\}}\mu(uv)\ge s^*-2D+2}$.

\item For each $v\in T^*$,  $d_{H^*}(v)\ge \frac{|N_{H^*}(v)|(s^*-2D+2)}{|N_{H^*}(v)|-1 }$.
\end{enumerate}
\end{claim}

\begin{proof}
\begin{itemize}
\item [(1)]
By Remark~\ref{remark-basic}, we have that
\[D-2\ge |E(H^*)\setminus(E(u) \cup E(v))| = |E(H^*)| - (d_{H^*}(u) + d_{H^*}(v)-
\mu(vw)).\]
So,
$d_{H^*}(u)+d_{H^*}(v)-\mu(uv)\ge s^*-D+2$.

\item [(2)] Suppose, to the contrary, that there is a vertex $v$ in $T^*$  such that $|N_{H^*}(v)|=1$. Let $N_{H^*}(v) = \{u\}$. Then, by (1), we have that
\[
d_{H^*}(u)+d_{H^*}(v)-\mu(uv)=d_{H^*}(u)\ge s^*-D+2> \frac{5}{2}D - 60 - D +2 > D,
\] a contradiction since each vertex in $H^*$ has degree at most $D$.

\item [(3)] Let $v\in T^*$ and $w\in N_{H^*}(v)$. By (1), we have that
\[d_{H^*}(v)+d_{H^*}(w)-\mu(vw)=d_{H^*}(w)+\sum_{u\in N_{H^*}(v)\setminus\{w\}}\mu(uv)\ge s^*-D+2.\] Since $\Delta(H^*)\le D$, we have that ${\displaystyle \sum_{u\in N_{H^*}(v)\setminus\{w\}}\mu(uv)\ge s^*-2D+2}$.

\item [(4)] Let $v\in T^*$. By (3),  we have that
$$\begin{aligned}|N_{H^*}(v)|(s^*-2D+2)&\le\sum_{w\in N_{H^*}(v)} \sum_{u\in N_{H^*}(v)\setminus\{w\}}\mu(uv)\\
& =(|N_{H^*}(v)|-1)\sum_{w\in N_{H^*}(v)}\mu(vw)\\
&=(|N_{H^*}(v)|-1)d_{H^*}(v).
 \end{aligned}$$

 Therefore, $d_{H^*}(v)\ge \frac{|N_{H^*}(v)|(s^*-2D+2)}{|N_{H^*}(v)|-1}$.
\end{itemize}\end{proof}
Note that when $|T^*|=6$, we have that the lower bound on the degrees of vertices in $H^*$ given by Claim~\ref{lem:degreeinH*}(4) is better than that in Claim~\ref{claim:degvinT} as $D\ge D_0$ is large enough.
Now, we prove that  $|V(H^*)| = 6$.

\begin{claim}\label{size>6}
$H^*$ is a connected graph with $|T^*|\le6$.
\end{claim}
\begin{proof}
First, we show that $H^*$ is a connected graph. Suppose that $H^*$ is disconnceted. By Claim~\ref{lem:degreeinH*}(2), $H^*$ has no isolated vertices. Thus, each component of $H^*$ contains at least one edge. Hence, there is a component $H'$ of $H^*$ which has at most $\frac{|E(H^*)|}{2}$ edges. Then there are at least $\frac{|E(H^*)|}{2}=\frac{s^*}{2}>D-2$ edges that are disjoint from edges in $H'$, contradicting Remark~\ref{remark-basic}.

Now, we show that $ |T^*|\le 6$. Suppose that $ |T^*|\ge 7$. Then there is a vertex 
$v\in T^*$ with
\[d_{H^*}(v)\le \frac{2|E(H^*)|}{|T^*|}\le\frac{2s^*}{7}<\frac{3(s^*-2D+2)}{2}<\frac{11}{14}D-60.\] Since $d_{H^*}(v)<\frac{3(s^*-2D+2)}{2}$, by Claim~\ref{lem:degreeinH*}(4), we have that $|N_{H^*}(v)|\ge 4$.

Let $A$ be the set of all vertices in $H^*$ that are of degree at least $\frac{11}{14}D-60$.
For each $w\in  N_{H^*}(v)$,  by Claim~\ref{lem:degreeinH*}(1), we have
\[d_{H^*}(w)\ge s^*-D+2-d_{H^*}(v)\ge s^*-D+2-\frac{2s^*}{7} > \frac{5}{7}\left(\frac{5}{2}D - 60 \right)  - D + 2  \geq \frac{11}{14}D-60.\]
It implies that $N_{H^*}(v) \subseteq A$. Thus, $|A| \geq |N_{H^*}(v)| \geq 4$. We proceed by show that if $|A|\ge 4$, then there is a vertex in $H^*$ with degree at most $\frac{2}{3}s^*-\frac{22}{21}D+80$ which is smaller than $\frac{2s^*}{7}$.

Note that by the handshaking lemma, we have $|A|<7$. Therefore, there is 
a vertex $v' \in T^*$(it may happen that $v'=v$) with 
\[\begin{aligned}d_{H^*}(v')&\le \frac{2|E(H^*)|-|A|(\frac{11}{14}D-60)}{|T^*|-|A|}\\
&\le\frac{2s^*-4(\frac{11}{14}D-60)}{3}\\
&=\frac{2}{3}s^*-\frac{22}{21}D+80\\
&<\frac{3(s^*-2D+2)}{2}\\
&<\frac{11}{14}D-60.
\end{aligned}
\]
Again, for each $w'\in  N_{H^*}(v')$,  by Claim~\ref{lem:degreeinH*}(1), we have that
\[d_{H^*}(w')\ge s^*-D+2-d_{H^*}(v')\ge s^*-D+2-\frac{2}{3}s^*+\frac{22}{21}D-80\ge \frac{s^*}{3}+\frac{D}{21}-78>\frac{11}{14}D-60.\]
Moreover, since $d_{H^*}(v')<\frac{3(s^*-2D+2)}{2}$, by Claim~\ref{lem:degreeinH*}(4), we have that $|N_{H^*}(v')|\ge 4$. Also, we deduce that $N_{H^*}(v') \subseteq A$. Note that by Claim~\ref{claim:degvinT}(2), we have that $d_{H^*}(v')\ge \frac{D}{2}-57$. Thus, there is a vertex $u\in T^*\setminus (A\cup \{v'\})$ with
\[\begin{aligned}d_{H^*}(u)&\le \frac{2s^*-4(\frac{s^*}{3}+\frac{D}{21}-78)-(\frac{D}{2}-57)}{2}\\
&\le \frac{1}{3}s^*-\frac{2}{21}D+2\times 78 -\frac{D}{4}+\frac{57}{2}\\
&\le \frac{1}{3}\times\frac{5}{2}D-\frac{2}{21}D-\frac{D}{4}+156 + \frac{57}{2} \\
&=\frac{41}{84}D+184.5  \\
&<\frac{D}{2}-57.
\end{aligned}\]
This is a contradiction to Claim~\ref{claim:degvinT}(2), and thus $|T^*| \leq 6$.
\end{proof}

\begin{claim}\label{size 6}
$H^*$ is a connected graph with $|T^*| =6$.
\end{claim}
\begin{proof}
        By Claim~\ref{size>6}, it suffices to show that  $|T^*|> 5$. Note that each vertex in $H^*$ is of degree at most $D$. So, we have that
$$D|T^*|\ge \sum_{v  \in T^*}d_{H^*}(v)=2s^*,
$$ which implies that $|T^*|\ge 5.$  Thus, $5 \leq |T^*| \leq 6$.

\medskip
Next, we show that $|T^*| = 6$.  We will show that there exists a contradiction if $|T^*| = 5$. From now on, we suppose that
$|T^*|= 5$.
Recall that  by Remark~\ref{rem:ScapV}, we have that $\sum_{w\in S^*}|N_{G}(w)\cap T^*|=\sum_{w\in S^*}|N_{G^*}(w)\cap T^*|=2s^*$. Thus,
\[\begin{aligned}\sum_{v\in T^*}d_G(v)&=\sum_{w\in S^*}|N_{G}(w)\cap T^*|+\sum_{w\in R^*\cup T^*}|N_{G}(v)\cap T^*|\\
&=2s^*+\sum_{w\in R^*\cup T^*}|N_{G}(v)\cap T^*|\\
&<5D,
\end{aligned}
\]
which implies that 
\begin{equation} \label{S-minus-S-start}
\sum_{w\in R^*\cup T^*}|N_{G}(v)\cap T^*|<5D-2s^*=2\left(\frac{5}{2}D-s^* \right)<2|S\setminus S^*|.
\end{equation}
If we can show that 
\[2|S\setminus S^*|\le\sum_{w\in R^*\cup T^*}|N_{G}(v)\cap T^*|,\] 
then we can conclude that $|T^*|=6$ since by inequality (\ref{S-minus-S-start}), we have the following contradiction 
\[2|S\setminus S^*|\le \sum_{w\in R^*\cup T^*}|N_{G}(v)\cap T^*|<2|S\setminus S^*|.
\]  
We proceed by showing an even stronger version that $|N_G(v)\cap T^*|\ge 2$ for each $v\in S\setminus S^*$. We formalize this as a subclaim.

\medskip

\begin{subclaim} \label{claim-T}
If $v\in S\setminus S^*$, then $|N_G(v)\cap T^*|\ge 2$.
\end{subclaim}
\begin{proof}
    Suppose that there is a vertex $v\in S\setminus S^*$ such that $|N_G(v)\cap T^*|\le 1$. Recall that Remark~\ref{rem:ScapV}, for each $u\in R^*$, we have that $|N_{G}(u)\cap S^*|=|N_{G^*}(u)\cap S^*|\le 2$. For $i\in [2]$, let $Y_i:=\{w\in N_{G}(v)\setminus T^*:~|N_{G}(w)\cap S^*|=i\}$. We claim that $|Y_2|\le 2(|T^*|+1)$.

Since $G$ is 2-degenerate, there are at most two vertices in $N_G(v)\setminus T^*$ that are after $v$ in the order $\sigma$.  We define $X_1$ and $X_2$ as follows.
\begin{eqnarray*}
X_1 &:= & \{w\in N_{G}(v)\setminus T^*: ~|N_{G}(w)\cap S^*|=2~\text{and}~v \text{~is later than}~w~\text{in the order}~\sigma\}, \mbox{ and } \\
X_2 & := & \{w\in N_{G}(v)\setminus T^*: ~|N_{G}(w)\cap S^*|=2~\text{and}~w \text{~is later than}~v~\text{in the order}~\sigma\}.  
\end{eqnarray*}
Then $Y_2=X_1\cup X_2$ and $|X_2|\le 2$. Thus, it suffices to show that $|X_1|\le 2|T^*|$.
 \begin{figure}[htbp]
  \begin{center}
  \includegraphics[scale=0.4]{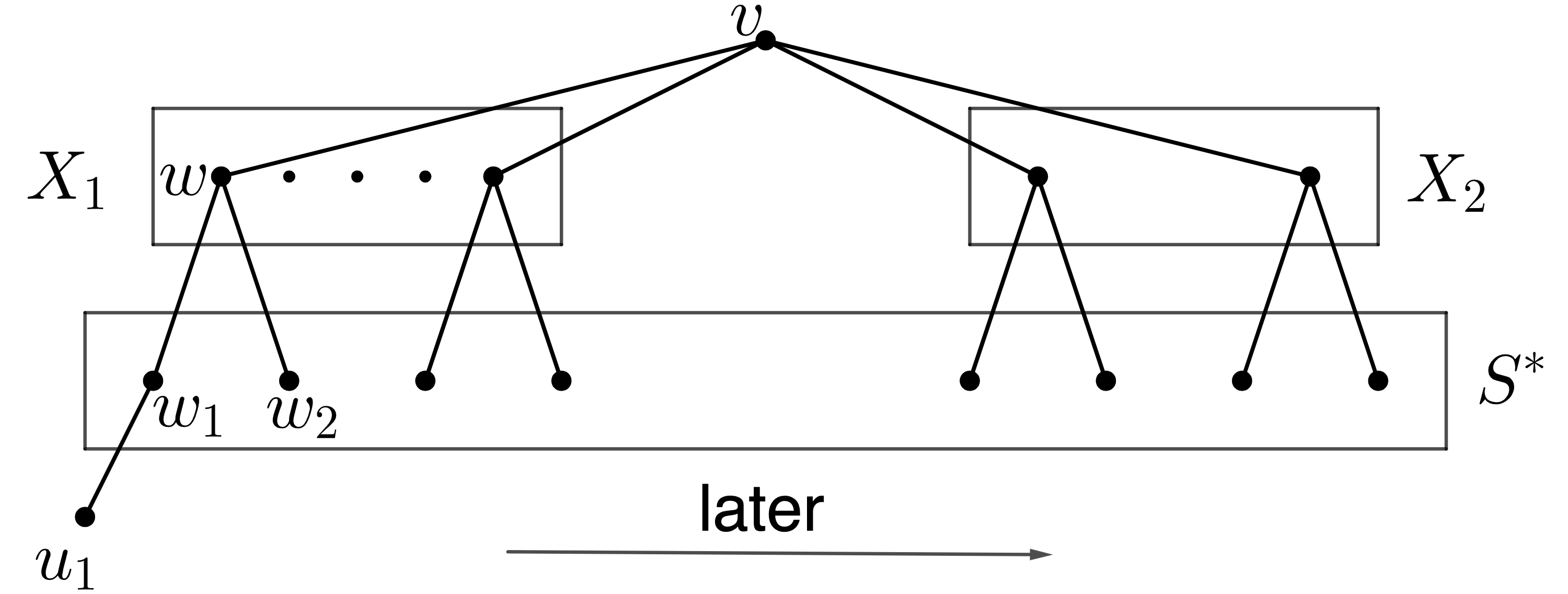}
  \end{center}
  \caption{The occurrences of partial vertices  of $G$ {\color{red}  w.r.t~$\sigma$.}}\label{Y2}
\end{figure}

For $w\in X_1$, let $\{w_1,w_2\}=N_{G^*}(w)\cap S^*$. Then by Remark~\ref{rem:ScapV} with $w\in R^*$, we have that $\{w_1,w_2\}= N_G(w)\cap S^*$.
Since $\sigma$ is a 2-degeneracy order of $G$ and $v$ is later than $w$ in $\sigma$, without loss of generality, say $w_1$ is before $w$ in the order $\sigma$. Since $w_1\in S^*$, we have that $|N_{G^*}(w_1)\cap T^*|=2$ by Claim~\ref{claim:degvinT} (1). Let  $\{u_1,u_2\}= N_{G^*}(w_1)\cap T^*$.  Again  by Remark~\ref{rem:ScapV} with $w_1\in S^*$, we have that $\{u_1,u_2\}= N_G(w_1)\cap T^*$.
Since $\sigma$ is a 2-degeneracy order of $G$ and $w$ is later than $w_1$ in $\sigma$, without loss of generality, say that $u_1$ is  before $w_1$ in the order $\sigma$ of $G$. Hence, for each $w\in X_1$, there is a path $P_w:=u_1w_1wv$ contained in $G$ where the occurrences of $u_1,w_1,w,v$ in the path $P_w$ are in the same order as their occurrences in the order $\sigma$ of $G$ where $u_1\in T^*$,  as shown in Figure~\ref{Y2}.

Since $\sigma$ is a 2-degeneracy order of $G$, for each vertex $x$ in $T^*$, there are at most two vertices later than $x$ in the order $\sigma$.  Thus, $|X_1| \le 2|T^*|$. Hence, $ |Y_2|\le |X_1|+|X_2|\le 2(|T^*|+1)$.

Since $v\in S\setminus S^*$ and $|N_{G}(v)\cap T^*|\le 1$, we have that  $$
\begin{aligned}
s^*&=|S^*|=|N_{G^2}(v)\cap S^*|\\
&= |N_G(v)\cap S^*|+|Y_1|+2|Y_2|+\left|\bigcup_{u\in N_{G}(v)\cap T^*}( N_{G}(u)\cap S^* ) \right|\\
&=(|N_G(v)\cap S^*|+|Y_1|+|Y_2|)+|Y_2|+\left|\bigcup_{u\in N_{G}(v)\cap T^*}( N_{G}(u)\cap S^* ) \right|\\
&\le d_{G}(v)+|Y_2|+\left|\bigcup_{u\in N_{G}(v)\cap T^*}( N_{G}(u)\cap S^* ) \right|\\
\end{aligned}$$
$$\begin{aligned}
&\le d_{G}(v)+2(|T^*|+1)+ \left|\bigcup_{u\in N_{G}(v)\cap T^*}( N_{G}(u)\cap S^*)  \right|\quad\quad\quad\quad\quad~\\
&\le D+14+D\\
&<\frac{5}{2}D-60,
\end{aligned}$$
 a contradiction.  This completes the proof of Claim \ref{claim-T}.
 \end{proof}

Therefore, $|T^*|=6$. This completes the proof of Claim \ref{size 6}.
\end{proof}

\subsection{$H^*_0$ is a complete graph if $|T^*|=6$} \label{subsection-two}
In this subsection, we prove that if the size of $T^*$ is 6, then the graph $H^*_0$ is a complete graph. This helps us to further analyze the degree of each vertex in the graph $H^*_0$. We first establish two useful claims, Claim~\ref{lem:degreeV=6} and Claim~\ref{claim:edgesum}, which are respectively about the minimum degree of the  graph $H^*_0$ and the sum of multiplicities of edges between some special vertices of $H^*$. Then, by using these two claims and constructing some multigraphs based on $H^*_0$, we prove that $\bar{H^*_0}$ is triangle free in Claim~\ref{lem:H0trianglegfree}. Lastly, we finish the proof that $H^*_0$ is a complete graph in Claim~\ref{lem:H0complete}. \\

Let $T^*=\{v_1,\ldots,v_6\}$ in this  subsection and the next subsection.\\

To enhance the clarity of the proof of Claim~\ref{lem:degreeV=6}, we begin by providing a general outline before delving into the formal proof. We assume that $N_{H^*_0}(v_1) = \{v_2, v_6\}$. By Claim~\ref{lem:degreeinH*}(4), we deduce that $d_{H^*}(v_1)\ge 2s^*-4D+4$, which is close to $D$. Consequently, the adjacencies between vertices in $T^*\setminus{v_1}$ are somewhat constrained. First, we show that the set $\{v_3, v_4, v_5\}$ cannot form an independent set in $H^*_0$ when $d_{H^*_0}(v_1)=2$. Next, we show that $\{v_3, v_4, v_5\}$ induces a triangle in $H^*_0$ under this condition. And then, we finish the proof in Claim \ref{lem:H0complete}.

\begin{Claim}\label{lem:degreeV=6}
For each $v\in T^*$, $d_{H^*_0}(v)\ge 3$.
\end{Claim}
\begin{proof}
Without loss of generality, by Claim~\ref{lem:degreeinH*}(2), we may assume that  $d_{H^*_0}(v_1)=2$ and $N_{H^*_0}(v_1)=\{v_2,v_6\}$. Then by Claim~\ref{lem:degreeinH*}(4), $d_{H^*}(v_1)\ge 2s^*-4D+4$.  We first prove that  $\{v_3v_4,v_4v_5,v_3v_5\}\subseteq E(H^*_0)$.

Suppose that $\{v_3,v_4,v_5\}$ is an independent set. Then  $d_{H^*_0}(v_i)=2$ for each $i\in \{3,4,5\}$. Thus, by Claim~\ref{lem:degreeinH*}(4),  for each $i\in \{3,4 ,5\}$ and $j\in \{2,6\}$, $v_iv_j\in E(H^*_0)$ and $d_{H^*}(v_i)\ge 2s^*-4D+4$. Hence,
from $\sum_{v \in T^*}d_{H^*} (v) = 2{s^*}$, we have that
\begin{equation} \label{equation-v2v6}
d_{H^*}(v_2)+d_{H^*}(v_6)\le 2s^*-4(2s^*-4D+4) = 16D - 6s^* - 16 < s^*-D+2.
\end{equation}
 Recall that by Claim~\ref{lem:degreeinH*}(1), for each $uv\in E(H^*)$, we have that $d_{H^*}(u)+d_{H^*}(v)-\mu(uv)\ge s^*-D+2$.
Since $d_{H^*}(v_2)+d_{H^*}(v_6)< s^*-D+2$,
we have that $v_2v_6\notin E(H^*_0)$ which implies that $d_{H^*_0}(v_2)=d_{H^*_0}(v_6)=4$. On the other hand, from inequality~(\ref{equation-v2v6}) above,  there exists $i\in \{2,6\}$ such that $d_{H^*}(v_i) \leq 8D - 3s^* - 8$. Since $d_{H^*_0}(v_i)=4$, by Claim~\ref{lem:degreeinH*}(4), $d_{H^*}(v_i) \geq \frac{4(s^* - 2D+2)}{3}$, which is a contradiction 
since $\frac{4(s^* - 2D+2)}{3} > 8D - 3s^* - 8$.

\medskip
Therefore, $\{v_3,v_4,v_5\}$ is not an independent set in $H^*_0$. So, without loss of generality, assume that $v_3v_4\in E(H^*_0)$.  Then by Remark~\ref{remark-basic},
$$\mu(v_5v_2)+\mu(v_5v_6)+\mu(v_2v_6)+d_{H^*}(v_1)\le D-2.$$
Thus,
\begin{equation}\label{equa:v526}
\mu(v_5v_2)+\mu(v_5v_6)\le D-2-d_{H^*}(v_1)\le  D-2-(2s^*-4D+4).
\end{equation}
Observe that $d_{H^*_0}(v_5)\le 4$ since  $v_5v_1\notin E(H^*_0)$. Thus, by Claim~\ref{lem:degreeinH*}(4),  we have that $d_{H^*}(v_5)\ge \frac{4}{3}(s^*-2D+2)$.

Since $\mu(v_5v_2)+\mu(v_5v_6)\le  D-2-(2s^*-4D+4)<  \frac{4}{3}(s^*-2D+2)<d_{H^*}(v_5)$,
we have that $N_{H^*_0}(v_5)\cap \{v_3,v_4\}\neq\emptyset$. Without loss of generality, say $v_5v_4\in E(H^*_0)$. Then by Remark~\ref{remark-basic}, we have that 
$$\mu(v_3v_2)+\mu(v_3v_6)+\mu(v_2v_6)+d_{H^*}(v_1)\le D-2.$$
Thus,
\begin{equation}\label{equa:v236}
\mu(v_3v_2)+\mu(v_3v_6)+\mu(v_2v_6)\le  D-2-(2s^*-4D+4).
\end{equation}
If $v_3v_5\notin E(H^*_0)$, then by inequalities~(\ref{equa:v526}) and ~(\ref{equa:v236}) 
$$\begin{aligned}
&d_{H^*}(v_4)+d_{H^*}(v_1)+2(5D-2s^*-6) \\
&=d_{H^*}(v_4)+d_{H^*}(v_1)+2(D-2-(2s^*-4D+4)) \\
&\ge  d_{H^*}(v_4)+d_{H^*}(v_1)+\mu(v_2v_3)+\mu(v_3v_6)+\mu(v_2v_6)+\mu(v_5v_2)+\mu(v_5v_6)\\
&=s^*>{\frac{5}{2}D-60}.\\
\end{aligned}
$$
Thus, $d_{H^*}(v_4)>D$ or $d_{H^*}(v_1)>D$ since $2(5D-2s^*-6) < 2(5D - 2 (\frac{5}{2}D - 60) - 6) = 228$.  This is a contradiction.
Therefore, $\{v_3v_4,v_4v_5,v_5v_3\}\subseteq E(H^*_0)$.

\medskip
Since $d_{H^*_0}(v_1)=2$ and $\{v_3v_4,v_4v_5,v_5v_3\}\subseteq E(H^*_0)$, by Remark~\ref{remark-basic}, we have that $\mu(v_2v_6)+d_{H^*}(v_1)\le D-2$ which implies that $\mu(v_2v_6)\le D-2-d_{H^*}(v_1)\le 5D-2s^*-6$.

By Claim~\ref{lem:degreeinH*}(1) with $v_1v_2\in E(H^*_0)$, we have that $$\begin{aligned}\mu(v_2v_3)+\mu(v_2v_4)+\mu(v_2v_5)&\ge s^*-D+2-d_{H^*}(v_1)-\mu(v_2v_6)\\
&\ge s^*-D+2-(D-2)\\
&>\frac{D}{2}-60.
\end{aligned}$$
Then, without loss of generality, say $$\mu(v_2v_3)\ge \frac{\frac{D}{2}-60}{3}>\frac{D}{7}.$$
It implies that
$$\mu(v_1v_6)+\mu(v_1v_2)+\mu (v_2v_3)=d_{H^*}(v_1)+\mu(v_2v_3)\ge 2s^*-4D+4+\frac{D}{7}>D-2,$$
which leads to a contradiction since all of these edges are not adjacent with $v_4v_5$ in $H^*$.
Therefore, each vertex in $H^*_0$ has degree at least 3.
\end{proof}

\begin{Claim}\label{claim:edgesum}\cite{CY}
If $v,w,x,y \in T^*$ are distinct vertices and $vw,xy\in E(H^*)$, then $\mu(vx)+\mu(vy)+\mu(wx)+\mu(wy)\ge s^*-2D+2$.
\end{Claim}
\begin{proof}
Let $E_1:=\{e\in E(H^*): e\text{~does~not~share~any~endpoint~with}~vw\}$, $E_2:=\{e\in E(H^*): e\text{~does~not~share~any~endpoint~with}~xy\}$ and $E_3:=E(H^*)
\setminus(E_1\cup E_2)$. Note that for each edge $e\in E_3$, $e$ has one endpoint in $\{v,w\}$ and the other endpoint in $\{x,y\}$. Thus, $|E_3|=\mu(wx)+\mu(vx)+\mu(wy)+\mu(vy)$.
Since $v,w,x,y \in T^*$ are distinct vertices and $vw,xy\in E(H^*)$, by 
Remark~\ref{remark-basic}, we have that $|E_1|\le D-2$ and $|E_2|\le D-2$. Therefore, $|E_3|=|E(H^*)\setminus(E_1\cup E_2)|\ge |E(H^*)|-|E_1|-|E_2|\ge s^*-2(D-2)>s^*-2D+2$.
\end{proof}

\begin{remark} \label{remark-claim37} \rm
Note that if there are distinct vertices $v,w,x,y \in T^*$ with $vw,xy\in E(H^*)$, and $vx,vy,wx\notin E(H^*)$, then by Claim~\ref{claim:edgesum}, we have that $\mu(wy)\ge s^*-2D+2>\frac{5}{2}D-60-2D+2> \frac{D}{2}-60$. This helps us to show that the degree of $w$ in $H^*$ is large.
\end{remark}

\begin{Claim}\label{lem:H0trianglegfree}
$\bar{H^*_0}$ is triangle free.
\end{Claim}
\begin{proof}
Observe that $\bar{H^*_0}$ contains at most two triangles since each vertex of $H^*_0$ has degree at least three. Suppose that $\bar{H^*_0}$ contains a triangle. Then it suffices to consider the following three cases.

\medskip
{\bf Case 1} $\bar{H^*_0}$ consists of two disjoint triangles.

Let $H_1:=H^*_0$, as shown in Figure~\ref{H0trianglefree}(1).  By Claim~\ref{lem:degreeinH*}(1), for each edge $vw\in E(H_1)$, we have that $d_{H^*}(v)+d_{H^*}(w)-\mu(vw)\ge s^*-(D-2)$. Summing over all $9$ edges of $H_1$ gives
\begin{equation}\label{equa:degreesum}
\sum_{vw\in E(H_1)}(d_{H^*}(v)+d_{H^*}(w)-\mu(vw))\ge \sum_{vw\in E(H_1)}(s^*-(D-2)).
\tag{$\mathcal{A}$}
\end{equation}

It is straightforward to check that each edge of $H^*$ is counted exactly 5 times. Thus, the inequality can be rewritten as $$5s^*\ge 9(s^*-(D-2)),$$
which simplifies to $s^*\le \frac{9}{4}(D-2)<\frac{5}{2}D-60$, a contradiction.

\begin{figure}
\begin{center}
\begin{tikzpicture}[u/.style={fill=black,minimum size =4pt,ellipse,inner sep=1pt},node distance=1.5cm,scale=0.9]
\node[u] (v1) at (120:1.5){};
\node[u] (v2) at (60:1.5){};
\node[u] (v3) at (0:1.5){};
\node[u] (v4) at (300:1.5){};
\node[u] (v5) at (240:1.5){};
\node[u] (v6) at (180:1.5){};

  \draw (v1) -- (v2);
  \draw[dashed]  (v1) -- (v3);
  \draw (v1) -- (v4);
  \draw (v1) -- (v5);
  \draw (v2) -- (v3);
  \draw[dashed]  (v2) -- (v4);
  \draw[dashed]  (v2) -- (v5);
  \draw (v3) -- (v4);
  \draw (v3) -- (v5);
  \draw[dashed]  (v4) -- (v5);
  \draw[dashed]  (v1) -- (v6);
  \draw (v2) -- (v6);
  \draw[dashed]  (v3) -- (v6);
   \draw (v4) -- (v6);
  \draw (v5) -- (v6);

 \node[left=0.1cm,font=\small] at (v1) {\bf3};
    \node[above=0.01cm,font=\small] at (v1) {$v_1$};
  \node[right=0.1cm,font=\small,font=\bfseries\color{black}] at (v2) {\bf 3};
   \node[above=0.01cm,font=\small] at (v2) {$v_2$};
   \node[right=0.001,font=\small] at (v3) {$v_3$};
   \node[right=0.05cm,below=0.001cm,font=\small,font=\bfseries\color{black}] at (v3) {\bf 3};
  \node[below=0.01cm,font=\small] at (v4) {$v_4$};
  \node[right=0.01cm,font=\small,font=\bfseries\color{black}] at (v4) {\bf 3};
    \node[below=0.01cm,font=\small] at (v5) {$v_5$};
  \node[left=0.01cm,font=\small,font=\bfseries\color{black}] at (v5) {\bf 3};
  \node[left=0.01cm,font=\small] at (v6) {$v_6$};
  \node[left=0.05cm,below=0.0015cm,font=\small,font=\bfseries\color{black}] at (v6) {\bf 3};   \node[left=0.7cm,below=0.6cm] at (v4) {(1)};
 \end{tikzpicture}
 \begin{tikzpicture}[u/.style={fill=black,minimum size =4pt,ellipse,inner sep=1pt},node distance=1.5cm,scale=0.9]
\node[u] (v1) at (120:1.5){};
\node[u] (v2) at (60:1.5){};
\node[u] (v3) at (0:1.5){};
\node[u] (v4) at (300:1.5){};
\node[u] (v5) at (240:1.5){};
\node[u] (v6) at (180:1.5){};

  \draw (v1) -- (v2);
  \draw[bend left=20] (v1) to  (v2);
  \draw[dashed]  (v1) -- (v3);
  \draw (v1) -- (v4);
  \draw (v1) -- (v5);
  \draw (v2) -- (v3);
  \draw[bend left=20] (v2) to  (v3);
  \draw[dashed]  (v2) -- (v4);
  \draw[dashed]  (v2) -- (v5);
  \draw (v3) -- (v4);
  \draw (v3) -- (v5);
  \draw (v4) -- (v5);
  \draw[dashed]  (v1) -- (v6);
  \draw (v2) -- (v6);
  \draw[dashed]  (v3) -- (v6);
   \draw (v4) -- (v6);
  \draw (v5) -- (v6);

 \node[left=0.1cm,font=\small] at (v1) {\bf4};
    \node[above=0.01cm,font=\small] at (v1) {$v_1$};
  \node[right=0.1cm,font=\small,font=\bfseries\color{black}] at (v2) {\bf 5};
   \node[above=0.01cm,font=\small] at (v2) {$v_2$};
   \node[right=0.001,font=\small] at (v3) {$v_3$};
   \node[below=0.001cm,font=\small,font=\bfseries\color{black}] at (v3) {\bf 4};
  \node[below=0.01cm,font=\small] at (v4) {$v_4$};
  \node[right=0.01cm,font=\small,font=\bfseries\color{black}] at (v4) {\bf 4};
    \node[below=0.01cm,font=\small] at (v5) {$v_5$};
  \node[left=0.01cm,font=\small,font=\bfseries\color{black}] at (v5) {\bf 4};
  \node[left=0.01cm,font=\small] at (v6) {$v_6$};
  \node[left=0.05cm,below=0.0015cm,font=\small,font=\bfseries\color{black}] at (v6) {\bf 3};
   \node[left=0.7cm,below=0.6cm] at (v4) {(2)};
 \end{tikzpicture}
 \begin{tikzpicture}[u/.style={fill=black,minimum size =4pt,ellipse,inner sep=1pt},node distance=1.5cm,scale=0.2]
\node[u] (v1) at (120:6.5){};
\node[u] (v2) at (60:6.5){};
\node[u] (v3) at (0:6.5){};
\node[u] (v4) at (300:6.5){};
\node[u] (v5) at (240:6.5){};
\node[u] (v6) at (180:6.5){};

  \draw (v1) -- (v2);
  \draw[dashed]  (v1) -- (v3);
  \draw (v1) -- (v4);
  \draw (v1) -- (v5);
  \draw (v2) -- (v3);
  \draw (v2) -- (v4);
  \draw[dashed]  (v2) -- (v5);
  \draw (v3) -- (v4);
  \draw (v3) -- (v5);
  \draw (v4) -- (v5);
  \draw[dashed]  (v1) -- (v6);
  \draw (v2) -- (v6);
  \draw[dashed]  (v3) -- (v6);
   \draw (v4) -- (v6);
  \draw (v5) -- (v6);

\node[left=0.1cm,font=\small] at (v1) {\bf3};
    \node[above=0.01cm,font=\small] at (v1) {$v_1$};
  \node[right=0.1cm,font=\small,font=\bfseries\color{black}] at (v2) {\bf 4};
   \node[above=0.01cm,font=\small] at (v2) {$v_2$};
   \node[right=0.001,font=\small] at (v3) {$v_3$};
   \node[right=0.05cm,below=0.001cm,font=\small,font=\bfseries\color{black}] at (v3) {\bf 3};
  \node[below=0.01cm,font=\small] at (v4) {$v_4$};
  \node[right=0.01cm,font=\small,font=\bfseries\color{black}] at (v4) {\bf 5};
    \node[below=0.01cm,font=\small] at (v5) {$v_5$};
  \node[left=0.01cm,font=\small,font=\bfseries\color{black}] at (v5) {\bf 4};
  \node[left=0.01cm,font=\small] at (v6) {$v_6$};
  \node[left=0.05cm,below=0.0015cm,font=\small,font=\bfseries\color{black}] at (v6) {\bf 3};
  \node[left=0.7cm,below=0.6cm] at (v4) {(3)};
 \end{tikzpicture}
 \begin{tikzpicture}[u/.style={fill=black,minimum size =4pt,ellipse,inner sep=1pt},node distance=1.5cm,scale=0.2]
\node[u] (v1) at (120:6.5){};
\node[u] (v2) at (60:6.5){};
\node[u] (v3) at (0:6.5){};
\node[u] (v4) at (300:6.5){};
\node[u] (v5) at (240:6.5){};
\node[u] (v6) at (180:6.5){};

  \draw (v1) -- (v2);
  \draw[dashed]  (v1) -- (v3);
  \draw (v1) -- (v4);
  \draw (v1) -- (v5);
  \draw (v2) -- (v3);
  \draw (v2) -- (v4);
  \draw (v2) -- (v5);
  \draw (v3) -- (v4);
  \draw (v3) -- (v5);
  \draw (v4) -- (v5);
  \draw[dashed]  (v1) -- (v6);
  \draw (v2) -- (v6);
  \draw[dashed]  (v3) -- (v6);
   \draw (v4) -- (v6);
  \draw (v5) -- (v6);

 \node[left=0.1cm,font=\small] at (v1) {\bf3};
    \node[above=0.01cm,font=\small] at (v1) {$v_1$};
  \node[right=0.1cm,font=\small,font=\bfseries\color{black}] at (v2) {\bf 5};
   \node[above=0.01cm,font=\small] at (v2) {$v_2$};
   \node[right=0.001,font=\small] at (v3) {$v_3$};
   \node[right=0.05cm,below=0.001cm,font=\small,font=\bfseries\color{black}] at (v3) {\bf 3};
  \node[below=0.01cm,font=\small] at (v4) {$v_4$};
  \node[right=0.01cm,font=\small,font=\bfseries\color{black}] at (v4) {\bf 5};
    \node[below=0.01cm,font=\small] at (v5) {$v_5$};
  \node[left=0.01cm,font=\small,font=\bfseries\color{black}] at (v5) {\bf 5};
  \node[left=0.01cm,font=\small] at (v6) {$v_6$};
  \node[left=0.05cm,below=0.0015cm,font=\small,font=\bfseries\color{black}] at (v6) {\bf 3};
  \node[left=0.7cm,below=0.6cm] at (v4) {(4)};
 \end{tikzpicture}
 \end{center}
 \caption{The graphs $H^*_0$ and $H_1$ in the proof of Claim~\ref{lem:H0trianglegfree}}\label{H0trianglefree}
 \end{figure}
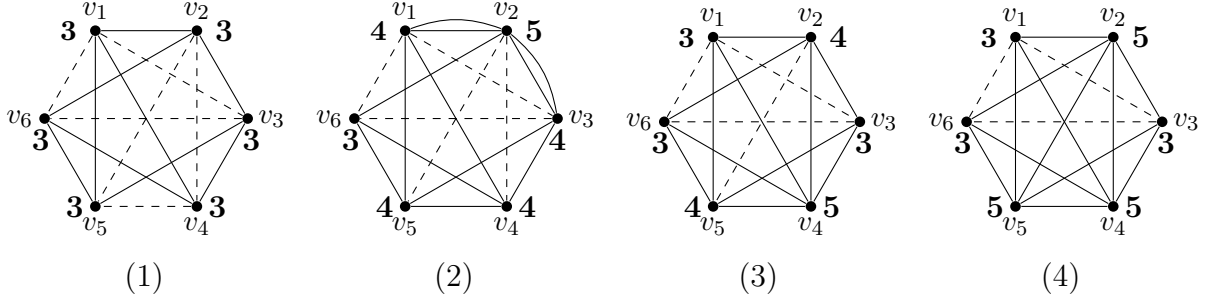
\medskip
{\bf Case 2} $\bar{H^*_0}$ consists of a disjoint union of a triangle and a path of length 2.

Let $H_1$ be the multigraph obtained from $H^*_0$ by adding edges, as shown in Figure~\ref{H0trianglefree}(2).  The inequality~\eqref{equa:degreesum}, now, can be rewritten as $$7s^*\ge 12(s^*-(D-2)),$$  since $H_1$ has 12 edges and each edge of $H^*$ is counted at most 7 times.  Therefore, $s^*\le \frac{12}{5}(D-2)<\frac{5}{2}D-60$, a contradiction.

\medskip
{\bf Case 3} $\bar{H^*_0}$ consists of a disjoint union of a triangle and a path of length 1  or $\bar{H^*_0}$ consists of a triangle.

It suffices to consider that $H^*_0$ is as shown in Figure~\ref{H0trianglefree}(3) or \ref{H0trianglefree}(4). Note that $d_{H^*_0}(v_1)=d_{H^*_0}(v_3)=d_{H^*_0}(v_6)=3$. Then, by Claim~\ref{lem:degreeinH*}(4),  $d_{H^*}(v_i)\ge \frac{3}{2}(s^*-2D+2)$ for each $i\in \{1,3,6\}$.

Since for $\{i,j,k\}=\{1,3,6\}$ each edge in $\{v_iv_2,v_iv_4,v_jv_2,v_jv_4,v_2v_4\}$ is disjoint from $v_5v_k$, by Remark~\ref{remark-basic}, we have that

\[
\begin{aligned} \label{D-minus-two}
\mu(v_iv_2)+\mu(v_iv_4)+\mu(v_jv_2)+\mu(v_jv_4)+\mu(v_2v_4)\le D-2.
\end{aligned}
\tag{$\mathcal{B}$}
\]

Observe that, by Claim~\ref{lem:degreeinH*}(3), $\mu(v_iv_2)+\mu(v_iv_4)\ge s^*-2D+2$ and $\mu(v_jv_2)+\mu(v_jv_4)\ge s^*-2D+2$.
Thus,  by inequality (\ref{D-minus-two}) above, we have that

\[
\begin{aligned}\label{upper-two-mu}
\mu(v_iv_2)+\mu(v_iv_4)&\le D-2-(\mu(v_jv_2)+\mu(v_jv_4))\\
& \le D-2-(s^*-2D+2)=3D-s^*-4, \mbox{ and }\\
\mu(v_2v_4)&\le D-2-(\mu(v_iv_2)+\mu(v_iv_4)+\mu(v_jv_2)+\mu(v_jv_4))\\& \le D-2-2(s^*-2D+2)=5D-2s^*-6.
\end{aligned}
\tag{$\mathcal{C}$}
\]

By the same argument, for $\{i,j,k\}=\{1,3,6\}$, since each edge in $\{v_iv_4,v_iv_5,v_kv_4,v_kv_5,v_4v_5\}$ is disjoint from $v_2v_j$ and each edge in $\{v_iv_2,v_iv_5,v_jv_2,v_jv_5,v_2v_5\}$ is disjoint from $v_4v_k$, we have that
\begin{eqnarray*}
\mu(v_iv_4)+\mu(v_iv_5)+\mu(v_kv_4)+\mu(v_kv_5)+\mu(v_4v_5) & \le &  D-2, \mbox{ and } \\
\mu(v_iv_2)+\mu(v_iv_5)+\mu(v_jv_2)+\mu(v_jv_5)+\mu(v_2v_5) & \le &  D-2.
\end{eqnarray*}

Again, from Claim~\ref{lem:degreeinH*}(3),  we have the following inequalities which are the same type as (\ref{upper-two-mu}).
\begin{eqnarray*}
s^*-2D+2   \le \mu(v_iv_4)+\mu(v_iv_5)  & \le &  D-2-(s^*-2D+2)=3D-s^*-4, \\
 \mu(v_4v_5) & \le & D-2-2(s^*-2D+2)=5D-2s^*-6, \\
s^*-2D+2   \le \mu(v_iv_2)+\mu(v_iv_5)  & \le &  D-2-(s^*-2D+2)=3D-s^*-4, \\
\mu(v_2v_5) & \le & D-2-2(s^*-2D+2)=5D-2s^*-6.
\end{eqnarray*}

Next we claim that $\mu(v_iv_p)\ge \frac{3}{2}s^*-\frac{7}{2}D+3$ for each $p\in \{2,4,5\}$. Suppose not; then for each $q\in \{2,4,5\}\setminus\{p\}$, by claim~\ref{lem:degreeinH*}(3), we have that 
$$\mu(v_iv_q)\ge s^*-2D+2-\mu(v_iv_p)\ge s^*-2D+2-\left(\frac{3}{2}s^*-\frac{7}{2}D+3\right)=\frac{3}{2}D-\frac{s^*}{2}-1$$
which implies that for $\{q,\ell\}=\{2,4,5\}\setminus\{p\}$, we have hat
\[
\mu(v_iv_q)+\mu(v_iv_\ell)\ge 2\left(\frac{3}{2}D-\frac{s^*}{2}-1\right)=3D-s^*-2>3D-s^*-4.
\]
This is a contradiction to inequality (\ref{upper-two-mu}).  Hence, $\mu(v_iv_p)\ge \frac{3}{2}s^*-\frac{7}{2}D+3$ for each $p\in \{2,4,5\}$.

\medskip
Therefore, by the symmetry of $i,j,k$,  for each $s\in \{1,3,6\}$ and $t\in \{2,4,5\}$, we have that \[\mu(v_sv_t)\ge \frac{3}{2}s^*-\frac{7}{2}D+3.\]

From the upper bounds on $\mu(v_2v_5), \ \mu(v_4v_5),$ and  $ \mu(v_2v_4)$  in inequality (\ref{upper-two-mu}),
we have that
\[\mu(v_2v_5)+\mu(v_4v_5)+\mu(v_2v_4)\le 3 \left(D-2-2(s^*-2D+2)\right)=3(5D-2s^*-6).\]
So, $$d_{H^*}(v_1)+d_{H^*}(v_3)+d_{H^*}(v_6)\ge s^*-3(5D-2s^*-6)=7s^*-15D+18.$$
Then there are integers
$s\in \{1,3,6\}$ and $t\in \{2,4,5\}$ such that
\[ d_{H^*}(v_s)\ge \frac{1}{3}(7s^*-15D+18) \mbox{ and } \mu(v_sv_t)\ge \frac{1}{9}(7s^*-15D+18).
\]
Then, for each $r\in \{2,4,5\}\setminus\{t\}$, we have that 
\[\mu(v_sv_t)+\mu(v_sv_r)\ge  \frac{1}{9}(7s^*-15D+18)+\frac{3}{2}s^*-\frac{7}{2}D+3 >3D-s^*-4.
\]
This is a contradiction to inequality (\ref{upper-two-mu}).

Therefore, $\bar{H^*_0}$ is triangle free.
\end{proof}


 Now, we prove that $H^*_0$ is actually a complete graph.
\begin{Claim}\label{lem:H0complete}
$H^*_0$ is a complete graph.
\end{Claim}
\begin{proof}
Note that it is impossible for $\bar{H^*_0}$ to contain a star $K_{1,t}$ where $t\ge 3$ since each vertex of $H^*_0$ has degree at least three. We start the proof by asserting that there is no paths and cycles of length at least 4 in $\bar{H^*_0}$.
Suppose, to the contrary, that there is a path of length 4 in  $\bar{H^*_0}$.
Without loss of generality, let $v_1v_6v_5v_4v_3$ be a path of length 4 contained in $\bar{H^*_0}$. Then $v_2v_5\in E(H^*_0)$  by Claim~\ref{lem:degreeV=6}. Consider the edges $v_1v_5, v_6v_4$ and $v_3v_5,v_6v_4$. Then
by Claim~\ref{claim:edgesum}, we have that
\[\mu(v_1v_4)\ge s^*-2D+2 \mbox{ and } \mu(v_3v_6)\ge s^*-2D+2.
\]
So, by Claim~\ref{lem:degreeinH*}(3), we have that $d_{H^*}(v_1),$ $d_{H^*}(v_4),d_{H^*}(v_3),$ $d_{H^*}(v_6)\ge2(s^*-2D+2)=2s^*-4D+4$.
Then 
\[
d_{H^*}(v_2)+d_{H^*}(v_5)\le 2s^*-4(2s^*-4D+4)=16D-6s^*-16<s^*-D+2,
\] 
contradicting to Claim~\ref{lem:degreeinH*}(1) since $v_2v_5\in E(H^*_0)$.
It also implies that there is no paths and cycles of length at least 5. Thus, to prove the assertion, it is sufficient to show that $\bar{H}^*_0$ does not  contain $C_4$ as a subgraph. Suppose, to the contrary, that $v_3v_4v_5v_6v_3$ is a cycle of length 4 contained in $\bar{H}^*_0$. Then by
Claim~\ref{claim:edgesum} with edges $v_3v_5$ and $v_4v_6$ contained in  $H^*$, we have the following contradiction.
\[
0 = \mu(v_3 v_4) + \mu(v_4 v_5) + \mu(v_5 v_6)+ \mu(v_3 v_6) \geq s^* - 2D + 2 > 0.
\]
Thus, the assertion holds. Therefore, by Claim~\ref{lem:H0trianglegfree} and  what we have proved above, there is no cycles contained in $\bar{H}^*_0$ and the paths contained in $\bar{H}^*_0$ is of length at most 3.

Now, we will show that there is no path of length at least 1 in $\bar{H}^*_0$. Suppose, to the contrary,  that there is a path of length at least 1 in $\bar{H}^*_0$. Then, according to the length of paths and the number of vertex disjoint paths contained in $\bar{H}^*_0$,  all the possible cases needed to consider are that $\bar{H}^*_0$ consists of
\begin{enumerate}
\item[(1)] a disjoint union of a path of length 3 and a path of length 1,
\item[(2)] a path of length 3,
\item[(3)] 2 vertex disjoint paths of length 2,

\item[(4)] a disjoint union of a path of length 2 and a path of length 1,
\item[(5)] a path of length 2,
\item[(6)] 3 vertex disjoint paths of length 1,
\item[(7)] 2 vertex disjoint paths of length 1, or
\item[(8)] a path of length 1.
 \end{enumerate}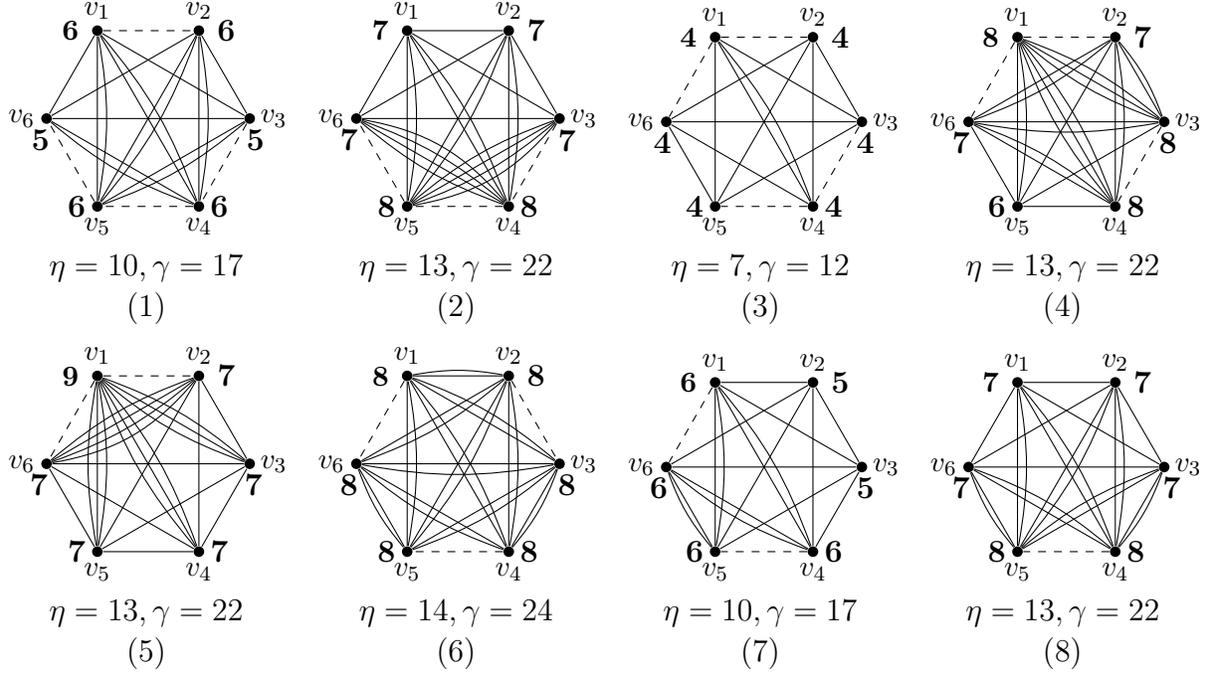
\begin{figure}[htbp]
\begin{center}
\begin{tikzpicture}[u/.style={fill=black,minimum size =4pt,ellipse,inner sep=1pt},node distance=1.5cm,scale=0.9]
\node[u] (v1) at (120:1.5){};
\node[u] (v2) at (60:1.5){};
\node[u] (v3) at (0:1.5){};
\node[u] (v4) at (300:1.5){};
\node[u] (v5) at (240:1.5){};
\node[u] (v6) at (180:1.5){};

  \draw[dashed] (v1) -- (v2);
  \draw (v1) -- (v3);
  \draw (v1) -- (v4);
  \draw[bend left=10] (v1) to  (v4);
  \draw (v1) -- (v5);
    \draw[bend left=10] (v1) to  (v5);
  \draw (v2) -- (v3);
  \draw (v2) -- (v4);
    \draw[bend left=10] (v2) to  (v4);
      \draw[bend left=10] (v2) to  (v5);
  \draw (v2) -- (v5);
  \draw[dashed] (v3) -- (v4);
  \draw (v3) -- (v5);
    \draw[bend left=10] (v3) to  (v5);
  \draw[dashed] (v4) -- (v5);
  \draw (v1) -- (v6);
  \draw (v2) -- (v6);
  \draw  (v3) -- (v6);
   \draw (v4) -- (v6);
    \draw[bend left=10] (v4) to  (v6);
  \draw[dashed] (v5) -- (v6);

 \node[left=0.1cm,font=\small] at (v1) {\bf6};
    \node[above=0.01cm,font=\small] at (v1) {$v_1$};
  \node[right=0.1cm,font=\small,font=\bfseries\color{black}] at (v2) {\bf 6};
   \node[above=0.01cm,font=\small] at (v2) {$v_2$};
   \node[right=0.001,font=\small] at (v3) {$v_3$};
   \node[right=0.08cm,below=0.001cm,font=\small,font=\bfseries\color{black}] at (v3) {\bf 5};
  \node[below=0.01cm,font=\small] at (v4) {$v_4$};
  \node[right=0.01cm,font=\small,font=\bfseries\color{black}] at (v4) {\bf 6};
    \node[below=0.01cm,font=\small] at (v5) {$v_5$};
  \node[left=0.01cm,font=\small,font=\bfseries\color{black}] at (v5) {\bf 6};
  \node[left=0.01cm,font=\small] at (v6) {$v_6$};
  \node[left=0.08cm,below=0.0015cm,font=\small,font=\bfseries\color{black}] at (v6) {\bf 5};
  \node[left=0.7cm,below=0.5cm] at (v4) {$\eta=10,\gamma=17$};

  \node[left=0.7cm,below=1cm] at (v4) {(1)};
 \end{tikzpicture}
 \begin{tikzpicture}[u/.style={fill=black,minimum size =4pt,ellipse,inner sep=1pt},node distance=1.5cm,scale=0.9]
\node[u] (v1) at (120:1.5){};
\node[u] (v2) at (60:1.5){};
\node[u] (v3) at (0:1.5){};
\node[u] (v4) at (300:1.5){};
\node[u] (v5) at (240:1.5){};
\node[u] (v6) at (180:1.5){};

  \draw (v1) -- (v2);
  \draw (v1) -- (v3);
  \draw (v1) -- (v4);
   \draw[bend left=10] (v1) to  (v4);
  \draw (v1) -- (v5);
   \draw[bend left=10] (v1) to  (v5);
  \draw (v2) -- (v3);
  \draw (v2) -- (v4);
   \draw[bend left=10] (v2) to  (v4);
  \draw  (v2) -- (v5);
   \draw[bend left=10] (v2) to  (v5);
  \draw[dashed] (v3) -- (v4);
  \draw (v3) -- (v5);
   \draw[bend left=10] (v3) to  (v5);
    \draw[bend left=20] (v3) to  (v5);
     \draw[bend right=10] (v3) to  (v5);
  \draw[dashed] (v4) -- (v5);
  \draw (v1) -- (v6);
  \draw (v2) -- (v6);
  \draw (v3) -- (v6);
   \draw (v4) -- (v6);
    \draw[bend left=10] (v4) to  (v6);
     \draw[bend right=10] (v4) to  (v6);
      \draw[bend right=20] (v4) to  (v6);
  \draw[dashed] (v5) -- (v6);

 \node[left=0.1cm,font=\small] at (v1) {\bf7};
    \node[above=0.01cm,font=\small] at (v1) {$v_1$};
  \node[right=0.1cm,font=\small,font=\bfseries\color{black}] at (v2) {\bf 7};
   \node[above=0.01cm,font=\small] at (v2) {$v_2$};
   \node[right=0.01,font=\small] at (v3) {$v_3$};
   \node[right=0.1cm,below=0.01cm,font=\small,font=\bfseries\color{black}] at (v3) {\bf 7};
  \node[below=0.01cm,font=\small] at (v4) {$v_4$};
  \node[right=0.01cm,font=\small,font=\bfseries\color{black}] at (v4) {\bf 8};
    \node[below=0.01cm,font=\small] at (v5) {$v_5$};
  \node[left=0.01cm,font=\small,font=\bfseries\color{black}] at (v5) {\bf 8};
  \node[left=0.01cm,font=\small] at (v6) {$v_6$};
  \node[left=0.1cm,below=0.0015cm,font=\small,font=\bfseries\color{black}] at (v6) {\bf 7};
  \node[left=0.7cm,below=0.5cm] at (v4) {$\eta=13,\gamma=22$};
   \node[left=0.7cm,below=1cm] at (v4) {(2)};
 \end{tikzpicture}
 \begin{tikzpicture}[u/.style={fill=black,minimum size =4pt,ellipse,inner sep=1pt},node distance=1.5cm,scale=0.2]
\node[u] (v1) at (120:6.5){};
\node[u] (v2) at (60:6.5){};
\node[u] (v3) at (0:6.5){};
\node[u] (v4) at (300:6.5){};
\node[u] (v5) at (240:6.5){};
\node[u] (v6) at (180:6.5){};

  \draw[dashed] (v1) -- (v2);
  \draw  (v1) -- (v3);
  \draw (v1) -- (v4);
   \draw[bend left=10] (v1) to  (v4);
  \draw (v1) -- (v5);
  \draw (v2) -- (v3);
  \draw (v2) -- (v4);
  \draw (v2) -- (v5);
  \draw[dashed] (v3) -- (v4);
  \draw (v3) -- (v5);
  \draw[dashed] (v4) -- (v5);
  \draw[dashed] (v1) -- (v6);
  \draw (v2) -- (v6);
  \draw (v3) -- (v6);
   \draw (v4) -- (v6);
  \draw  (v5) -- (v6);

\node[left=0.1cm,font=\small] at (v1) {\bf4};
    \node[above=0.01cm,font=\small] at (v1) {$v_1$};
  \node[right=0.1cm,font=\small,font=\bfseries\color{black}] at (v2) {\bf 4};
   \node[above=0.01cm,font=\small] at (v2) {$v_2$};
   \node[right=0.001,font=\small] at (v3) {$v_3$};
   \node[right=0.05cm,below=0.001cm,font=\small,font=\bfseries\color{black}] at (v3) {\bf 4};
  \node[below=0.01cm,font=\small] at (v4) {$v_4$};
  \node[right=0.01cm,font=\small,font=\bfseries\color{black}] at (v4) {\bf 4};
    \node[below=0.01cm,font=\small] at (v5) {$v_5$};
  \node[left=0.01cm,font=\small,font=\bfseries\color{black}] at (v5) {\bf 4};
  \node[left=0.01cm,font=\small] at (v6) {$v_6$};
  \node[left=0.05cm,below=0.0015cm,font=\small,font=\bfseries\color{black}] at (v6) {\bf 4};
  \node[left=0.7cm,below=0.5cm] at (v4) {$\eta=7,\gamma=12$};
  \node[left=0.7cm,below=1cm] at (v4) {(3)};
 \end{tikzpicture}
 \begin{tikzpicture}[u/.style={fill=black,minimum size =4pt,ellipse,inner sep=1pt},node distance=1.5cm,scale=0.2]
\node[u] (v1) at (120:6.5){};
\node[u] (v2) at (60:6.5){};
\node[u] (v3) at (0:6.5){};
\node[u] (v4) at (300:6.5){};
\node[u] (v5) at (240:6.5){};
\node[u] (v6) at (180:6.5){};

  \draw[dashed]  (v1) -- (v2);
  \draw (v1) -- (v3);
  \draw (v1) -- (v4);
  \draw (v1) -- (v5);
  \draw (v2) -- (v3);
  \draw (v2) -- (v4);
  \draw (v2) -- (v5);
  \draw[dashed]  (v3) -- (v4);
  \draw (v3) -- (v5);
  \draw (v4) -- (v5);
  \draw[dashed]  (v1) -- (v6);
  \draw (v2) -- (v6);
  \draw (v3) -- (v6);
   \draw (v4) -- (v6);
  \draw (v5) -- (v6);
  \draw[bend left=10] (v1) to  (v3);
  \draw[bend right=10] (v1) to  (v3);
  \draw[bend left=10] (v1) to  (v4);
  \draw[bend right=10] (v1) to  (v4);
  \draw[bend left=10] (v1) to  (v5);
  \draw[bend left=10] (v2) to  (v3);
    \draw[bend left=10] (v2) to  (v4);
      \draw[bend left=10] (v2) to  (v6);
  \draw[bend left=10] (v3) to  (v6);
  \draw[bend right=10] (v4) to  (v6);

 \node[left=0.1cm,font=\small] at (v1) {\bf8};
    \node[above=0.01cm,font=\small] at (v1) {$v_1$};
  \node[right=0.1cm,font=\small,font=\bfseries\color{black}] at (v2) {\bf 7};
   \node[above=0.01cm,font=\small] at (v2) {$v_2$};
   \node[right=0.001,font=\small] at (v3) {$v_3$};
   \node[right=0.05cm,below=0.001cm,font=\small,font=\bfseries\color{black}] at (v3) {\bf 8};
  \node[below=0.01cm,font=\small] at (v4) {$v_4$};
  \node[right=0.01cm,font=\small,font=\bfseries\color{black}] at (v4) {\bf 8};
    \node[below=0.01cm,font=\small] at (v5) {$v_5$};
  \node[left=0.01cm,font=\small,font=\bfseries\color{black}] at (v5) {\bf 6};
  \node[left=0.01cm,font=\small] at (v6) {$v_6$};
  \node[left=0.1cm,below=0.0015cm,font=\small,font=\bfseries\color{black}] at (v6) {\bf 7};
  \node[left=0.7cm,below=0.5cm] at (v4) {$\eta=13,\gamma=22$};
  \node[left=0.7cm,below=1cm] at (v4) {(4)};
 \end{tikzpicture}

 \begin{tikzpicture}[u/.style={fill=black,minimum size =4pt,ellipse,inner sep=1pt},node distance=1.5cm,scale=0.9]
\node[u] (v1) at (120:1.5){};
\node[u] (v2) at (60:1.5){};
\node[u] (v3) at (0:1.5){};
\node[u] (v4) at (300:1.5){};
\node[u] (v5) at (240:1.5){};
\node[u] (v6) at (180:1.5){};

  \draw[dashed] (v1) -- (v2);
  \draw (v1) -- (v3);
  \draw[bend left=10] (v1) to  (v3);
  \draw[bend right=10] (v1) to  (v3);
  \draw (v1) -- (v4);
  \draw[bend left=10] (v1) to  (v4);
  \draw[bend right=10] (v1) to  (v4);
  \draw (v1) -- (v5);
    \draw[bend left=10] (v1) to  (v5);
    \draw[bend right=10] (v1) to  (v5);
  \draw (v2) -- (v3);
  \draw (v2) -- (v4);
  \draw (v2) -- (v5);
  \draw (v3) -- (v4);
  \draw (v3) -- (v5);
  \draw (v4) -- (v5);
  \draw[dashed] (v1) -- (v6);
  \draw (v2) -- (v6);
  \draw[bend left=10] (v2) to  (v6);
  \draw[bend left=20] (v2) to  (v6);
  \draw[bend right=10] (v2) to  (v6);
  \draw  (v3) -- (v6);
   \draw (v4) -- (v6);
  \draw (v5) -- (v6);

 \node[left=0.1cm,font=\small] at (v1) {\bf9};
    \node[above=0.01cm,font=\small] at (v1) {$v_1$};
  \node[right=0.1cm,font=\small,font=\bfseries\color{black}] at (v2) {\bf 7};
   \node[above=0.01cm,font=\small] at (v2) {$v_2$};
   \node[right=0.001,font=\small] at (v3) {$v_3$};
   \node[right=0.05cm,below=0.001cm,font=\small,font=\bfseries\color{black}] at (v3) {\bf 7};
  \node[below=0.01cm,font=\small] at (v4) {$v_4$};
  \node[right=0.01cm,font=\small,font=\bfseries\color{black}] at (v4) {\bf 7};
    \node[below=0.01cm,font=\small] at (v5) {$v_5$};
  \node[left=0.01cm,font=\small,font=\bfseries\color{black}] at (v5) {\bf 7};
  \node[left=0.01cm,font=\small] at (v6) {$v_6$};
  \node[left=0.1cm,below=0.0015cm,font=\small,font=\bfseries\color{black}] at (v6) {\bf 7};
  \node[left=0.7cm,below=0.5cm] at (v4) {$\eta=13,\gamma=22$};
  \node[left=0.7cm,below=1cm] at (v4) {(5)};
 \end{tikzpicture}
\begin{tikzpicture}[u/.style={fill=black,minimum size =4pt,ellipse,inner sep=1pt},node distance=1.5cm,scale=0.9]
\node[u] (v1) at (120:1.5){};
\node[u] (v2) at (60:1.5){};
\node[u] (v3) at (0:1.5){};
\node[u] (v4) at (300:1.5){};
\node[u] (v5) at (240:1.5){};
\node[u] (v6) at (180:1.5){};

  \draw (v1) -- (v2);
   \draw[bend left=10] (v1) to  (v2);
  \draw (v1) -- (v3);
   \draw[bend left=10] (v1) to  (v3);
  \draw (v1) -- (v4);
   \draw[bend left=10] (v1) to  (v4);
  \draw (v1) -- (v5);
   \draw[bend left=10] (v1) to  (v5);
  \draw[dashed] (v2) -- (v3);
  \draw (v2) -- (v4);
   \draw[bend left=10] (v2) to  (v4);
  \draw  (v2) -- (v5);
   \draw[bend left=10] (v2) to  (v5);
  \draw (v3) -- (v4);
   \draw[bend left=10] (v3) to  (v4);
  \draw (v3) -- (v5);
   \draw[bend left=10] (v3) to  (v5);
  \draw[dashed] (v4) -- (v5);
  \draw[dashed] (v1) -- (v6);
  \draw (v2) -- (v6);
   \draw[bend left=10] (v2) to  (v6);
  \draw (v3) -- (v6);
   \draw[bend left=10] (v3) to  (v6);
   \draw (v4) -- (v6);
    \draw[bend left=10] (v4) to  (v6);
    \draw (v5) -- (v6);
    \draw[bend left=10] (v5) to  (v6);

 \node[left=0.1cm,font=\small] at (v1) {\bf8};
    \node[above=0.01cm,font=\small] at (v1) {$v_1$};
  \node[right=0.1cm,font=\small,font=\bfseries\color{black}] at (v2) {\bf 8};
   \node[above=0.01cm,font=\small] at (v2) {$v_2$};
   \node[right=0.001,font=\small] at (v3) {$v_3$};
   \node[right=0.1cm,below=0.001cm,font=\small,font=\bfseries\color{black}] at (v3) {\bf 8};
  \node[below=0.01cm,font=\small] at (v4) {$v_4$};
  \node[right=0.01cm,font=\small,font=\bfseries\color{black}] at (v4) {\bf 8};
    \node[below=0.01cm,font=\small] at (v5) {$v_5$};
  \node[left=0.01cm,font=\small,font=\bfseries\color{black}] at (v5) {\bf 8};
  \node[left=0.01cm,font=\small] at (v6) {$v_6$};
  \node[left=0.1cm,below=0.0015cm,font=\small,font=\bfseries\color{black}] at (v6) {\bf 8};
  \node[left=0.7cm,below=0.5cm] at (v4) {$\eta=14,\gamma=24$};
   \node[left=0.7cm,below=1cm] at (v4) {(6)};
 \end{tikzpicture}
 \begin{tikzpicture}[u/.style={fill=black,minimum size =4pt,ellipse,inner sep=1pt},node distance=1.5cm,scale=0.2]
\node[u] (v1) at (120:6.5){};
\node[u] (v2) at (60:6.5){};
\node[u] (v3) at (0:6.5){};
\node[u] (v4) at (300:6.5){};
\node[u] (v5) at (240:6.5){};
\node[u] (v6) at (180:6.5){};

  \draw (v1) -- (v2);
  \draw  (v1) -- (v3);
  \draw (v1) -- (v4);
   \draw[bend left=10] (v1) to  (v4);
  \draw (v1) -- (v5);
  \draw[bend left=10] (v1) to  (v5);
  \draw (v2) -- (v3);
  \draw (v2) -- (v4);
  \draw (v2) -- (v5);
  \draw (v3) -- (v4);
  \draw (v3) -- (v5);
  \draw[dashed] (v4) -- (v5);
  \draw[dashed] (v1) -- (v6);
  \draw (v2) -- (v6);
  \draw (v3) -- (v6);
   \draw (v4) -- (v6);
   \draw[bend left=10] (v4) to  (v6);
  \draw  (v5) -- (v6);
  \draw[bend left=10] (v5) to  (v6);

\node[left=0.1cm,font=\small] at (v1) {\bf6};
    \node[above=0.01cm,font=\small] at (v1) {$v_1$};
  \node[right=0.1cm,font=\small,font=\bfseries\color{black}] at (v2) {\bf 5};
   \node[above=0.01cm,font=\small] at (v2) {$v_2$};
   \node[right=0.001,font=\small] at (v3) {$v_3$};
   \node[right=0.05cm,below=0.001cm,font=\small,font=\bfseries\color{black}] at (v3) {\bf 5};
  \node[below=0.01cm,font=\small] at (v4) {$v_4$};
  \node[right=0.01cm,font=\small,font=\bfseries\color{black}] at (v4) {\bf 6};
    \node[below=0.01cm,font=\small] at (v5) {$v_5$};
  \node[left=0.01cm,font=\small,font=\bfseries\color{black}] at (v5) {\bf 6};
  \node[left=0.01cm,font=\small] at (v6) {$v_6$};
  \node[left=0.1cm,below=0.0015cm,font=\small,font=\bfseries\color{black}] at (v6) {\bf 6};
  \node[left=0.7cm,below=0.5cm] at (v4) {$\eta=10,\gamma=17$};
  \node[left=0.7cm,below=1cm] at (v4) {(7)};
 \end{tikzpicture}
 \begin{tikzpicture}[u/.style={fill=black,minimum size =4pt,ellipse,inner sep=1pt},node distance=1.5cm,scale=0.2]
\node[u] (v1) at (120:6.5){};
\node[u] (v2) at (60:6.5){};
\node[u] (v3) at (0:6.5){};
\node[u] (v4) at (300:6.5){};
\node[u] (v5) at (240:6.5){};
\node[u] (v6) at (180:6.5){};

  \draw (v1) -- (v2);
  \draw (v1) -- (v3);
  \draw (v1) -- (v4);
  \draw (v1) -- (v5);
  \draw (v2) -- (v3);
  \draw (v2) -- (v4);
  \draw (v2) -- (v5);
  \draw(v3) -- (v4);
  \draw (v3) -- (v5);
  \draw[dashed] (v4) -- (v5);
  \draw (v1) -- (v6);
  \draw (v2) -- (v6);
  \draw (v3) -- (v6);
   \draw (v4) -- (v6);
  \draw (v5) -- (v6);
  \draw[bend left=10] (v1) to  (v4);
  \draw[bend left=10] (v1) to  (v5);
  \draw[bend left=10] (v2) to  (v5);
  \draw[bend left=10] (v2) to  (v4);
  \draw[bend left=10] (v3) to  (v4);
  \draw[bend right=10] (v3) to  (v5);
    \draw[bend right=10] (v4) to  (v6);
        \draw[bend left=10] (v5) to  (v6);

 \node[left=0.1cm,font=\small] at (v1) {\bf7};
    \node[above=0.01cm,font=\small] at (v1) {$v_1$};
  \node[right=0.1cm,font=\small,font=\bfseries\color{black}] at (v2) {\bf 7};
   \node[above=0.01cm,font=\small] at (v2) {$v_2$};
   \node[right=0.001,font=\small] at (v3) {$v_3$};
   \node[right=0.1cm,below=0.001cm,font=\small,font=\bfseries\color{black}] at (v3) {\bf 7};
  \node[below=0.01cm,font=\small] at (v4) {$v_4$};
  \node[right=0.01cm,font=\small,font=\bfseries\color{black}] at (v4) {\bf 8};
    \node[below=0.01cm,font=\small] at (v5) {$v_5$};
  \node[left=0.01cm,font=\small,font=\bfseries\color{black}] at (v5) {\bf 8};
  \node[left=0.01cm,font=\small] at (v6) {$v_6$};
  \node[left=0.1cm,below=0.0015cm,font=\small,font=\bfseries\color{black}] at (v6) {\bf 7};
  \node[left=0.7cm,below=0.5cm] at (v4) {$\eta=13,\gamma=22$};
  \node[left=0.7cm,below=1cm] at (v4) {(8)};
 \end{tikzpicture}
 \end{center}
 \caption{\small The multigraph $H_1$ defined for cases (1)-(8) in the proof of Claim~ \ref{lem:H0complete} and the degree of each vertex in $H_1$ is listed.}\label{fig:H0complete}
 \end{figure}
Here, we prove that none of the cases happens by hiring the same argument as that used in the proof of Case 1 in Claim~\ref{lem:H0trianglegfree}. For each $i\in [8]$, let $H_1$ be a multigraph obtained from $H^*_0$ by adding edges, as shown in Figure~\ref{fig:H0complete}(i).
Furthermore, for each edge $vw\in E(H^*)$, let $c(vw)$ be the number of times that $vw$ is counted while summing over all the edges of $H_1$ in the left side of inequality~\eqref{equa:degreesum}. Then, $c(vw)=d_{H_1}(v)+d_{H_1}(w)-\mu_{H_1}(vw)$.
Let  $$\gamma:=|E(H_1)| ~\text{and}~ \eta:= \max_{vw\in E(H^*)}c(vw),$$ as shown in Figure~\ref{fig:H0complete} for each case. Then, inequality~\eqref{equa:degreesum}, now, can be rewritten as
$$\eta s^*\ge \gamma(s^*-(D-2)),$$  which can be simplified to $s^*\le \frac{\gamma}{\gamma-\eta}(D-2)<\frac{5}{2}D-60$, a contradiction.

Therefore, $H^*_0$ is a complete graph.
\end{proof}

\subsection{Final step of the Proof  of Theorem \ref{thm:Main}}\label{subsection-three}
In this subsection, we complete the proof of Theorem \ref{thm:Main}.   From the previous subsection, we know that $H^*_0$ is a complete graph. We first show that the multigraph $H^*$ is almost balanced.  That is, in Claim~\ref{claim:degH^*_0complete}, we show that
each vertex $v$ in $H^*$ has $\frac{5}{6}D - c_1 \leq d_{H^*}(v) \leq \frac{5}{6}D + c_1'$ and, in Claim~\ref{claim:sizeofVij}, we show that the multiplicity of each edge $uv$ in $H^*$ satisfies $\mu(uv) \geq \frac{D}{6}-c_2$ where $c_1$, $c_1'$ and $c_2$ are small enough
compared to $D$.
Next, by introducing  new terms, we compute the size of $S\setminus S^*$ in Claim~\ref{lower-D(ij)}--Claim~\ref{size of S(ij)}.
Finally, we finish the proof of Theorem~\ref{thm:Main} at the end of this subsection.

\medskip
First, we show the following claim.
\begin{Claim}\label{claim:degH^*_0complete}
Let $v\in T^*$. Then $\frac{25D-9s^*-50}{3}\ge d_{H^*}(v)\ge  \frac{3s^*-5D+10}{3}$.
\end{Claim}
\begin{proof}
Since $H^*_0$ is a complete graph, by Claim~\ref{lem:degreeinH*}(3) for distinct $i,j\in [6]$,
$$d_{H^*}(v_i)+d_{H^*}(v_j)-\mu(v_iv_j)\ge s^*-D+2.$$
Then
$$
\begin{aligned}&\left(\sum_{i\in [6]\setminus\{j\}}d_{H^*}(v_i)\right)+5d_{H^*}(v_j)-\sum_{i\in [6]\setminus\{j\}}\mu(v_iv_j)\\
&=\left(\sum_{v\in T^*}d_{H^*}(v)\right)+3d_{H^*}(v_j)\\
&\ge 5(s^*-D+2).
\end{aligned}$$
Thus, for each vertex $v_j$ in $H^*$,
\[d_{H^*}(v_j)\ge\frac{ 5(s^*-D+2)-2s^*}{3}=\frac{3s^*-5D+10}{3}.\]
Therefore, since $|T^*| = 6$, we have that
$$
d_{H^*}(v) =  2s^*- \sum_{u \in T^* \setminus \{v\}} d_{H^*} (u)
\le2s^*-5\frac{3s^*-5D+10}{3}\\
=\frac{25D-9s^*-50}{3}.
$$
\end{proof}
Note that Claim \ref{claim:degH^*_0complete} implies that for each vertex $v$ in $H^*$, we have  that $\frac{5}{6}D - 57 \leq d_{H^*}(v) \leq \frac{5}{6}D + 164$.

\begin{figure}[htbp]
\begin{center}
\begin{tikzpicture}
  \draw (0.6,0) circle (1.0cm);
  \node[below=1.0cm] at (0.6,0) {$R^*$};
  \node[right=0.4cm] at (0.6,0) {$W$};
  \draw (2,0) circle (1.0cm);
 \node[below=1.0cm] at (2,0) {$S$};
 \node at (2,0) {$S^*$};
  \draw (3.4,0) circle (1.0cm);
\node[below=1.0cm] at (3.4,0) {$T^*$};
\node[left=0.4cm] at (3.4,0) {$U$};
\end{tikzpicture}
\end{center}
\caption{The relations between sets $S$, $R^*$ and $T^*$.}\label{Fig:SR*T*}
\end{figure}
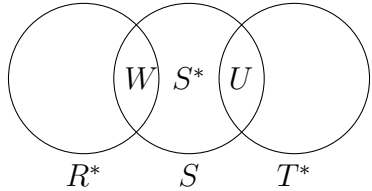

\medskip

Recall that $V(G)=V(G^*)=R^*\cup S^*\cup T^*$. Let $W:=S\cap R^*$ and $U:=S\cap T^*$, as shown in Figure~\ref{Fig:SR*T*}. Then $S=W\cup S^*\cup U$ and $W\cap S^*=U\cap S^*=\emptyset$.
We establish a claim which tells us  the adjacency between vertices in  $W\cup U$ and $T^*$ in the graph $G$.
\begin{Claim}\label{claim:UW}
\begin{itemize}
\item [(1)] Let $v\in U$. Then $2\le|N_G(v)\cap T^*|\le 5$.
\item [(2)] Let $v\in W$. Then $2\le|N_G(v)\cap T^*|\le 6$.
\end{itemize}
\end{Claim}
\begin{proof}
By Claim \ref{claim-T},   if $v \in S\setminus S^*$ and $|T^*|=6$, then $|N_G(v)\cap T^*|\ge 2$. Since $S\setminus S^*=W\cup U$ by the definition of $W$  and $U$, we have that (1) and (2) hold.
\end{proof}

\medskip
Let $U_i:=\{u\in U: |N_G(u)\cap T^*|=i\}$ and $\alpha_i:=|U_i|$ where $i\in [5]$. 
Observe that by Claim~\ref{claim:UW}, $\alpha_1=0$. Let $W_i:=\{u\in W: |N_G(u)\cap  T^*|=i+1\}$ and $\beta_i:=|W_i|$  where $i\in [5]$.
Then $$|S\setminus S^*|=\sum^5_{i=1}(\alpha_i+\beta_i)>\frac{5}{2}D-s^*.$$

\medskip
Next, we start to count the number of vertices in $S\setminus S^*$. We use double counting strategy. For this, we
define $D_{i,j}$, $S_{i, j}$, and $V_{i, j}$ as follows.

\begin{definition} \label{def-DSV} \rm
Let $T^* = \{v_1,v_2,v_3,v_4,v_5,v_6\}$. \\
\begin{enumerate}[(1)]
\item For $1\le i<j\le 6$, let \[D_{i,j}:=\{v\in S^*: N_{G^*}(v)\cap \{v_i,v_j\}=\emptyset\}\]  and let $|D_{i,j}|=d_{i,j}$.
Note that each vertex in $D_{i, j}$  corresponding to an edge in  $H^*$ that are incident to neither  $v_i$ nor $v_j$.  And note that $d_{i,j}\le D-2$ by Claim \ref{lem:degreeinH*} (1), since $H_0^*$ is a complete graph.
\item For $1\le i<j\le 6$, let \[
S_{i,j}:=\{v\in S \setminus (S^*\cup \{v_i,v_j\}): ~N_{G}(v)\cap \{v_i,v_j\} =\emptyset \} \]  and let $s_{i,j}:=|S_{i,j}|$.  Note that $S_{i,j}$ is the set of vertices in 
$S \setminus  (S^*\cup \{v_i,v_j\})$ that are adjacent to neither $v_i$ nor $v_j$.

\item
For each edge $v_iv_j\in E(H^*_0)$, let $V_{i,j}$ be the set of vertices in $S^*$ that correspond to edges between $v_i$ and  $v_j$. Then, by Construction~\ref{construction-H*}, we have that  $|V_{i,j}|=\mu(v_iv_j)$.
\end{enumerate}
\end{definition}

\medskip
As outlined above, we will show that $H^*$ is almost balanced.  By applying Claim~\ref{claim:degH^*_0complete}, we show that for every distinct $i,j\in [6]$,
$|V_{i,j}|=\mu(v_iv_j)\ge \frac{D}{6}-2000$.  This is formalized in the following claim.

\begin{claim}\label{claim:sizeofVij}For every distinct $i,j\in [6]$,
$|V_{i,j}|=\mu(v_iv_j)\ge \frac{87s^*-217D+374}{3}$.
 Furthermore, $|V_{i,j}|>\frac{D}{6}-2000$.
\end{claim}
\begin{proof}
Suppose that $\mu(v_iv_j)\le \frac{87s^*-217D+373}{3}$ for some  $i,j\in [6]$. Then, by the Pigeonhole principle, there exists $k\in[6]\setminus\{i,j\}$ such that $\mu(v_iv_k)\ge  \frac{d_{H^*}(v_i)-\frac{87s^*-217D+373}{3}}{4}$. Then by Claim~\ref{claim:degH^*_0complete}  and Remark~\ref{sizeofS*}, we have that
$$
\begin{aligned}
d_{H^*}(v_i)+d_{H^*}(v_k)-\mu(v_iv_k)&\le\frac{3}{4}d_{H^*}(v_i)+d_{H^*}(v_k)+\frac{87s^*-217D+373}{12}\\
&\le \frac{7}{4}\left(\frac{25D-9s^*-50}{3}\right)+\frac{87s^*-217D+373}{12}\\
&\le \frac{175D-63s^*-350+87s^*-217D+373}{12}\\
&\le \frac{24s^*-42D+23}{12} \\
&< s^*-D+2,
\end{aligned}$$
a contradiction.  Thus, for every distinct $i,j\in [6]$, we have that
$|V_{i,j}|=\mu(v_iv_j)\ge \frac{87s^*-217D+374}{3}$, and
$|V_{i,j}|\ge \frac{0.5D-87\times 60+374}{3}>\frac{D}{6}-2000$, since $s^*> \frac{5}{2}D-60$.
\end{proof}

\medskip
\begin{claim} \label{lower-D(ij)}
For $1\le i<j\le 6$, $d_{i,j}\ge 6s^*-14D+28$.
\end{claim}
\begin{proof}  Note that
$\sum_{1\le i<j\le 6}d_{i,j}=6|E(H^*)|=6s^*$. Since $d_{i',j'}\le D-2$ for each $1\le i'<j'\le 6$, we have that \[
d_{i,j} = 6s^* - \sum_{\{i',j'\} \neq \{i, j\}} d_{i',j'} \ge 6s^*-14(D-2) = 6s^*-14D+28.
\]
\end{proof}

Note that Claim \ref{lower-D(ij)} implies that for each $1\le i<j\le 6$, we have
$d_{i,j}\ge 6s^*-14D+28 > D - 332$, since $s^* >\frac{5}{2}D - 60$.
\medskip

\medskip
\begin{claim}
For each vertex $v\in V_{i,j}$, let $R_v:=\{w\in R^*\cap N_{G^*}(v): N_{G^*}(w)\cap D_{i,j}\neq\emptyset \}$. Then  we have that $|R_v|\ge d_{i,j}$.
\end{claim}

\begin{proof}
Since $v\in S^*$ and $S^*$ induces a complete subgraph of ${(G^*)}^2$, for each  $u\in D_{i,j}$, there is a vertex $w\in R^*$ such that $vw,wu\in E(G^*)$.  Recall that for each vertex $w\in R^*$, by Remark~\ref{rem:ScapV}, we have that $|N_{G^*}(w)\cap S^*|\le 2$. Thus, $|R_v|\ge d_{i,j}$ as $\{v\}\cup D_{i,j}\subseteq S^*$.
\end{proof}

\medskip

Next, we show a very important claim which plays central role in the proof.
We give an upper bound on the size of $S_{i, j}$ as follows.
Recall that $S_{i,j}$ is the set of vertices in $S \setminus (S^*\cup \{v_i,v_j\})$ that are adjacent to neither $v_i$ nor $v_j$, and $s_{i, j} = |S_{i,j}|$. 
The outline of the proof Claim \ref{size of S(ij)} is as follows.
We will show that $|S_{i,j}|$ is not too large.
If $|S_{i,j}|$ is large  enough, then there is a subset $Z$ of $S_{i,j}$ of size exactly $D-1-d_{i,j}$.  Then each vertex in $V_{i,j}$ should adjacent in $G^2$ to each vertex in $Z$. Then we show that there is some vertex $v$ in $V_{i,j}$ such that for each $z\in Z$, there is a vertex $w_z\in N_{G}(v)\setminus (R_v\cup \{v_i,v_j\})$ with $vw_z,zw_z\in E(G)$ where for distinct  $z_1,z_2\in Z$, we have that $w_{z_1}\neq w_{z_2}$. Then, along with $R_v$ and $v_i,v_j$, we have that the degree of $v$ in $G$ is at least $d_{i,j}+2+D-1-d_{i,j}>D$, a contradiction.

\begin{claim} \label{size of S(ij)}
For $1\le i<j\le 6$, we have that $s_{i,j}\le D-2-d_{i,j}$.
\end{claim}
\begin{proof}
Suppose that $s_{i,j}>D-2-d_{i,j}$ for some $i, j$. Let $Z\subseteq S_{i,j}$ with $|Z|=D-1-d_{i,j}$, then by Claim \ref{lower-D(ij)}, we have that
\[|Z|=D-1-d_{i,j} \le D-1-(6s^*-14D+28)\le6 \left(\frac{5}{2}D-s^* \right)-29<331.\]
Let \[V_0=\{u\in V_{i,j}:N_G(u)\cap Z\neq\emptyset\}.\] Since $Z\subseteq S_{i,j}$, by
Remark~\ref{rem:ScapV}, we have that
\begin{equation}\label{size-V-zero}
|V_0|\le 2Z < 2 \times 331.
\end{equation}
Then for each vertex in $v\in V_{i,j}\setminus V_0$ and each vertex $u\in Z$, there should be a vertex $w\in R^*\cup T^*$ such that $vw,uw\in E(G)$. Since for each vertex  $v$ in $V_{i,j}$, $N_G(v)\cap T^*=\{v_i,v_j\}$ and each vertex in $S_{i,j}$ is  adjacent to neither $v_i$ nor $v_j$, we have that $w\in R^*$. Additionally, recall that for $v\in V_{i,j}$ each vertex in $R_v$ has one neighbor $v$ and another neighbor in $D_{i,j}\subseteq S$. Thus, by Remark~\ref{rem:ScapV} with $w\in R^*$, we have that $w\in R^*\setminus R_v$.
For any pair of vertices $x,y\in Z$, let
\[W_{x,y}=\{w\in V(G):wx,wy\in E(G)~\text{and}~N_G(w)\cap (V_{i,j}\setminus V_0)\neq\emptyset\}.\]
For any $w\in W_{x,y}$,  we have that $|N_G(w)\cap (V_{i,j}\setminus V_0)|\le 2$ by Remark~\ref{rem:ScapV} with $w\in R^*$ and $V_{i,j}\subseteq S^*$.

Next we claim that $|W_{x,y}|\le 5$.
Suppose that $|W_{x,y}|\ge 6$. Let $B_{x,y}:=N_G(W_{x,y})\cap (V_{i,j}\setminus V_0)$. Then $|B_{x,y}|\ge 3$. Otherwise the subgraph of $G$ induced by $B_{x,y}\cup W_{x,y}\cup \{x,y\}$ is not 2-degenerate. But $|B_{x,y}|\ge 3$ implies that the subgraph of $G$ induced by $\{x,y,v_i,v_j\}\cup B_{x,y}\cup W_{x,y}$ is not 2-degenerate since every vertex in the induced subgraph has degree at least 3.
Thus, $|W_{x,y}|\le 5$ for any pair of vertices in $Z$.  Hence, \[\left|\bigcup_{x,y\in Z}W_{x,y}\right|\le \sum_{x,y\in Z}|W_{x,y}|\le 5\times{{|Z|}\choose 2}< 5\times {331 \choose 2}.\]
\begin{figure}[htbp]
  \begin{center}
  \includegraphics[scale=0.18]{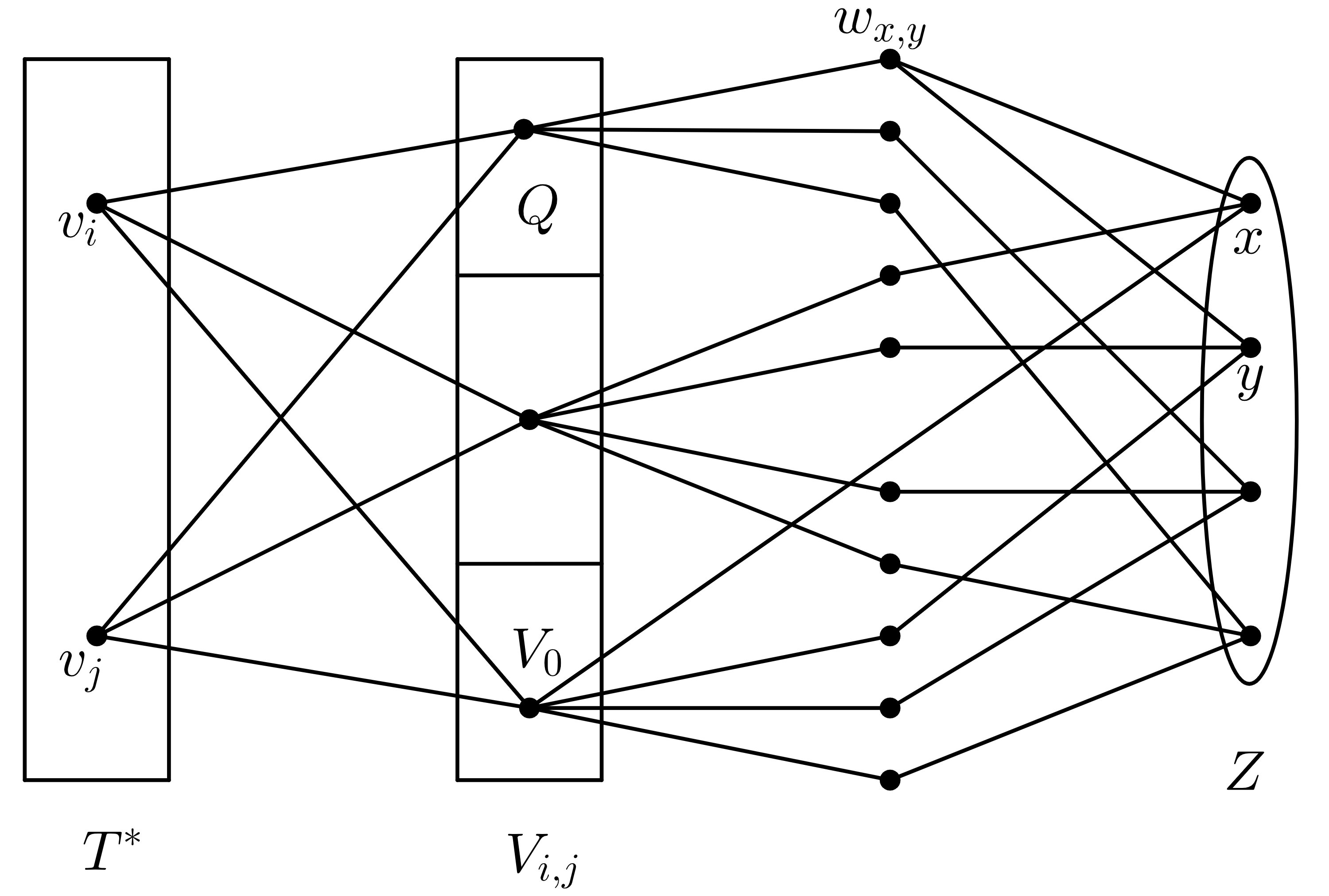}
  \end{center}
  \caption{ An example for adjacency in $G$ between vertices in $V_{i,j}$ and $Z$ }\label{fig:doublecounting}
\end{figure}

\noindent Let
\[Q=\{v\in V_{i,j}\setminus V_0: \text{there is a vertex $w\in N_G(v)\setminus\{v_i,v_j\}$ such that $|N_G(w)\cap Z|\ge 2$} \}.\]
Thus,  by Remark~\ref{rem:ScapV},
\begin{equation} \label{size-Q}
|Q|\le 2\left|\bigcup_{x,y\in Z}W_{x,y}\right|< 10\times{ 331\choose 2}.
\end{equation}

Now, we consider the subset $V_{i,j}\setminus (Q\cup V_0)$ (see Figure~\ref{fig:doublecounting}).
Note that  $|V_{i,j}|>\frac{D}{6}-2000$ from Claim~\ref{claim:sizeofVij} and
$D\ge D_0=6\times( 331\times 2+10\times{ 331\choose 2}+2000)$.
(Note that this is the only place where we need large value of $D_0$.)
So, by the inequalities  (\ref{size-V-zero}) and (\ref{size-Q}) we have that
\[
|V_{i,j}\setminus(Q\cup V_0)| \geq |V_{i,j}| - |Q| - |V_0| > 0.
\]
Thus the subset $V_{i,j}\setminus(Q\cup V_0)$ is nonempty.
Then for each vertex $v$ in $V_{i,j}\setminus(Q\cup V_0)$ and each vertex $x\in Z$, there is a vertex $w_{v,x}\in R^*\setminus R_v$ such that $vw_{v,x},xw_{v,x}\in E(G)$; furthermore,  for distinct $x,x'\in Z$, we have that $w_{v,x'}\neq w_{v,x}$ by the choice of $v$ in $V_{i,j}\setminus(Q\cup V_0)$.
 Then $d_G(v)\ge |R_v\cup \{v_i,v_j\}|+|Z|= d_{i,j}+2+|Z|>D$, a contradiction.
This completes the proof of Claim \ref{size of S(ij)}.
\end{proof}

\medskip
Hence, from Claim \ref{size of S(ij)}, we have the following property.
\begin{equation}\label{counting}
\begin{aligned}
\sum_{1\le i<j\le 6}s_{i,j}&\le \sum_{1\le i<j\le 6}(D-2-d_{i,j})\\
&= 15(D-2)-\sum_{1\le i<j\le 6}d_{i,j}\\&=15(D-2)-6|E(H^*)|\\
&=15D-6s^*-30\\
&=6\left(\frac{5}{2}D-s^*\right)-30.
\end{aligned}
\tag{$\mathcal{D}$}
\end{equation}

Note that the inequality (\ref{counting}) is very important in the final step of the proof.
Let $N_{G}[v]:=N_G(v)\cup \{v\}$. Next, we count how many times each vertex  in $S\setminus S^*$ appears on the left side of inequality~(\ref{counting}).

\begin{Counting} \label{count-spade}
Fix $i \in \{1, 2, 3, 4, 5\}$.  For each $v \in U_i \cup W_i$, we have $|T^* \setminus N_G[v]| = 5-i$, so $v$ is counted ${5-i \choose 2}$ times on the left side of inequality~(\ref{counting}).  Thus, we have the following.
\begin{enumerate}[(1)]
\item Each $v \in U_1\cup W_1$ is counted 6 times one the left side of inequality~(\ref{counting}).

\item Each $v \in U_2\cup W_2$ is counted 3 times one the left side of inequality~(\ref{counting}).

\item Each $v \in U_3\cup W_3$ is counted 1 time one the left side of inequality~(\ref{counting}).

\item Each $v \in U_4\cup W_4 \cup U_5\cup W_5$ is counted 0 time one the left side of inequality~(\ref{counting}).
\end{enumerate}
\end{Counting}
Therefore, from Counting \ref{count-spade} we have that
$$\sum_{1\le i<j\le 6}s_{i,j}=6(\alpha_1+\beta_1)+3(\alpha_2+\beta_2)+\alpha_3+\beta_3.$$

Recall that $|S\setminus S^*|=\sum^5_{i=1}(\alpha_i+\beta_i)$.
Since $|S\setminus S^*|>\frac{5}{2}D-s^*$, if we substitute $|S\setminus S^*|$ with
$\sum^5_{i=1}(\alpha_i+\beta_i)$ in inequality~(\ref{counting}), then we have that
$$6(\alpha_1+\beta_1)+3(\alpha_2+\beta_2)+\alpha_3+\beta_3<6\left(\sum_{i\in [5]}(\alpha_i+\beta_i)\right)-30.$$
Thus,
\begin{equation}\label{equa:mainineq}
30<3(\alpha_2+\beta_2)+5(\alpha_3+\beta_3)+6\sum_{i\in \{4,5\}}(\alpha_i+\beta_i).
\tag{$\mathcal{E}$}
\end{equation}

\medskip
 Let $J$ be the graph  obtained from the subgraph of $G$ induced by $(S\setminus S^*)\cup T^*$ by deleting edges that are not incident with any vertex in $T^*$. Observe that the graph $J$ is defined by inspiration from Counting~\ref{count-spade}.

\medskip
Recall that $U = S \cap T^*$.
Let $U^*=\{u^*:u\in U\}$. Then $|U^*|=|U|$. Let $J^*$ be the graph obtained from $J$ which is defined as follows.

$$\begin{aligned}
V(J^*)&=W\cup U^*\cup T^*,\\
E(J^*)&=\{ab\in E(H):a\in W,b\in T^*\} \cup \{u^*v:uv\in E(G),u\in U,v\in T^*\} \cup\\
~&~~~~\{u^*u:u\in U\}.
\end{aligned}
$$
Note that  for each $v\in W$, $|N_{J^*}(v)| = |N_{J}(v)|$ and for each $v\in U^*$, $|N_{J^*}(v)| = |N_{J}(v)\cap T^*| +1$.

\medskip
Next, we show that  actually  $J^*$ is a 2-degenerate bipartite graph.
 Finally, we use~Claim~\ref{claim:valueofxyz} to get a contradiction and finish our proof.

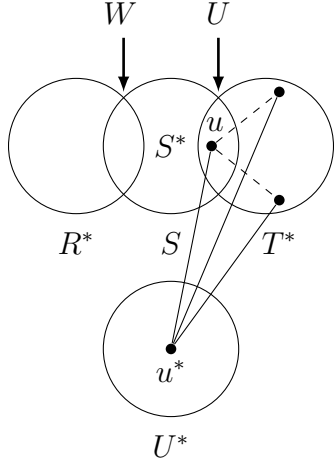
\begin{figure}[htbp]
\begin{center}
\begin{tikzpicture}[u/.style={fill=black,minimum size =4pt,ellipse,inner sep=1pt},node distance=1.5cm,scale=0.9]
\node[u] (v0) at (2.6,0){};
\node[u] (v1) at (3.6,0.8){};
\node[u] (v2) at (3.6,-0.8){};
\node[u] (v3) at (2,-3){};

  \node[below=1.0cm] at (0.6,0) {$R^*$};
  \node[above=0.4cm] at (1.25,1.2) {$W$};
  \node[above=0.05cm] at (2.65,0){$u$};
  \node[below=1.0cm] at (2,0) {$S$};
 \node at (2,0) {$S^*$};
 \node[below=1.0cm] at (3.6,0) {$T^*$};
\node[above=0.4cm] at (2.7,1.2) {$U$};
\node[below=0.01cm] at (2,-3) {$u^*$};
\node[below=1cm] at (2,-3) {$U^*$};
  \draw (2,0) circle (1.0cm);
  \draw (3.4,0) circle (1.0cm);
\draw[->, line width=1pt, >=latex, scale=2]  (0.65,0.8)--(0.65,0.4);
\draw[->, line width=1pt, >=latex, scale=2]  (1.35,0.8)--(1.35,0.4);
\draw (2,-3) circle (1.0cm);

\draw[dashed] (v0) to (v1);
\draw[dashed] (v0) to (v2);
  \draw (0.6,0) circle (1.0cm);
  \draw (v3) to (v1);
  \draw (v3) to (v2);
  \draw (v0) to (v3);

\end{tikzpicture}
\end{center}
\caption{An example of $J^*$ when $U\neq\emptyset$.}\label{Fig:J*}
\end{figure}

\medskip
\begin{claim}
The graph $J^*$ is a 2-degenerate bipartite graph.
\end{claim}
\begin{proof}
Observe that from the definition of the graph $J^*$,  we have that $J^*$ is a bipartite graph with bipartition $(W\cup U^*, T^*)$.  Now, we focus on the proof that $J^*$ is 2-degenerate.  Observe that if $U=\emptyset$, then by the definition of $J^*$, we have that $J^*$ is a subgraph of $J$ which means that $J^*$ is 2-degenerate.

It suffices to consider  the case when $U\neq\emptyset$ as shown in Figure~\ref{Fig:J*}. Note that, by the definition of $J^*$,  for $u\in U$ and $u^*\in U^*$ in $J^*$, we have that
\[\begin{aligned}N_{J}(u)&=(N_J(u)\cap W)\cup (N_J(u)\cap T^*)\\
&=(N_{J^*}(u)\cap W)\cup (N_{J^*}(u^*)\cap T^*)\setminus\{u\}\\
&=(N_{J^*}(u)\setminus\{u^*\})\cup (N_{J^*}(u^*)\setminus\{u\}),\\
\end{aligned}\]
where $(N_{J^*}(u)\setminus\{u^*\})\cap (N_{J^*}(u^*)\setminus\{u\})=\emptyset$.

Let $\sigma(J)$ be a 2-degeneracy order of $J$ where each vertex has at most two neighbors later.  To obtained a 2-degeneracy order $\sigma(J^*)$ of $J^*$ from $J$, we proceed as follows. For each $u\in U$, if there is a neighbor $w$ of $u$ in $W$ which is later than $u$ in $\sigma(J)$, then put $u^*$ directly before $u$ such that $u^*$ and $u$ are consecutive, otherwise, put $u^*$ directly after $u$ such that $u$ and $u^*$ are consecutive.
We call it the induced order of $J^*$, denoted  $\sigma(J^*)$.

Now, we prove that $\sigma(J^*)$ is a 2-degeneracy order of $J^*$. Note that for both cases, it suffices to check that $u$ and $u^*$ have at most two neighbors later in the order $\sigma(J^*)$.
For the former, since there is a neighbor $w$ of $u$ in $W$ which is later than $u$ in $\sigma$, by the above analysis, $u^*w\notin E(J^*)$.
Therefore, $u^*$ has at most $1+(2-1)=2$ neighbors later in $\sigma(J^*)$ while the number of neighbors of $u$ that are later does not increase.  For the latter, all the neighbors of $u$ that are later than $u$ in $\sigma(J)$ should be in $T^*$, then by the definition of $J^*$, such adjacency does not exist in $J^*$, therefore, $u$ has only one neighbor later in $\sigma(J^*)$ while there are at most 2 neighbors of $u^*$ inherited from $u$  later than $u^*$ in $\sigma(J^*)$.

Therefore, $J^*$ is a 2-degenerate bipartite graph.
\end{proof}

Define
$$
\begin{aligned}
x&:=|\{u:u\in W\cup U^*,|N_{J^*}(u)\cap T^*|=3\}|,\\
y&:=|\{u:u\in W\cup U^*,|N_{J^*}(u)\cap T^*|=4\}|,\\
z&:=|\{u:u\in W\cup U^*,|N_{J^*}(u)\cap T^*|\ge 5\}|.
\end{aligned}
$$
Then inequality~(\ref{equa:mainineq}) can be simplified as $$30<3x+5y+6z.$$
Since $J^*$ is a 2-degenerate bipartite graph, the following holds.
\begin{claim}\label{claim:valueofxyz}
\begin{itemize}
\item [(1)] $x+y+z\le 8$. Furthermore, if $x+y+z=8$, then $y=z=0$ and $x=8$.

\item [(2)] $y+z\le 4$. Furthermore, if $y+z=4$, then $z=0$ and $y=4$.

\item [(3)] $z\le 2$.
\end{itemize}
\end{claim}

\begin{proof}
Note that $J^*$ is a 2-degenerate bipartite graph. Let $\sigma(J^*)$ be a 2-degeneracy order of $J^*$ where each vertex has at most 2 vertices later. Let $\{u,v,w\}$ be the last three vertices in the order $\sigma(J^*)$. Then $|E(J^*[\{u,v,w\}])|\le 2$.

For $i\in [3]$,  let $Q_i:=\{v\in W\cup U^*:d_{J^*}(v)\ge i+2\}$. Let $J_i$ be the subgraph of $J^*$ induced by $Q_i\cup T^*$.  That is, $J_i:=J^*[Q_i\cup T^*]$.  Then $J_i$ is 2-degenerate
and  bipartite. From the observation $|E(J^*[\{u,v,w\}])|\le 2$, we have that for $i\in [3]$,
\[(i+2)|Q_i|\le\sum_{u\in Q_i}|N_{J_i}(u)\cap T^*|=\sum_{u\in Q_i}d_{J_i}(u)=|E(J_i)|
\le 2(|V(J_i)|-3)+2 \le 2(6+|Q_i|-3)+2,
\]
which simplifies to $|Q_i|\le 8/i$. Therefore, for $i=1$, we have that $x+y+z\le 8$; furthermore  equality holds only when $x=8$ and $y=z=0$. For $i=2$, we have that $y+z\le 4$; furthermore equality holds only when $y = 4$ and $z=0$.  For $i=3$, we have that  $z\le 2$.
\end{proof}
Thus, we have the following system of linear inequalities:
$$
\left\{ \begin{array}{lr}
3x+5y+6z>30; &\\
x+y+z\le 8; &\\
y+z\le 4; &\\
z\le 2;\\
x,y,z\ge 0.
\end{array}
\right.
$$
It is straight forward to check that its only integer solutions
are
$$(4,4,0),(5,2,1),(4,3,1),(5, 1,2),(3,2,2),(4,2,2).$$  However,
none of these solutions satisfies Claim~\ref{claim:valueofxyz}. It implies that the graph $J^*$ does not exist which leads to a contradiction.

This completes the proof of Theorem \ref{thm:Main}.
\qed

\section{Remark}

The following question was asked in \cite{CY}.
\begin{question} \label{Q-CY}
For each positive integer $k$, what is the minimum value $\alpha_k$ such that there exists
a constant $c_k$ such that if $G$ is $k$-degenerate with maximum degree at most $D$, then
$\omega(G^2) \leq \alpha_k D + c_k$?
\end{question}
If $G$ is $k$-degenerate with $\Delta(G) \leq D$, then $G^2$ has degeneracy at most
$(2k-1)D-k^2$, so $\alpha_k \leq 2k-1$.  In particular, for $k = 2$, Cranston and Yu \cite{CY} showed that $\alpha_2 = \frac{5}{2}$.  On the other hand, our result implies that $c_2 = 0$ when $D \geq D_0 = 6\times(331\times 2+10\times{ 331\choose 2}+2000)$.

\section*{Acknowledgments}
Seog-Jin Kim is supported by the National Research Foundation of Korea(NRF) grant funded by the Korea government(MSIT)(No.NRF-2021R1A2C1005785),
and Xiaopan Lian is supported by National Natural Science Foundation of China (No. 12161141006) and China Scholarship Council(CSC) Grant $\#$202306200010. 

\newpage
\section{Appendix (The Proof of Theorem \ref{thm:modify})}\label{appendix}
In this appendix, we present the proof Theorem \ref{thm:modify}. We shall emphasize here that the proof is exactly that of Theorem 5 in \cite{CY}. Theorem~\ref{thm:modify} explains how to obtain $(G^*,S^*,\sigma^*)$ from $(G,S,\sigma)$ such that $|S \setminus S^*| \leq 60$.

\begin{THMMAIN}(Theorem 5, \cite{CY})
For every positive integer $D \geq 1732$,
if a 2-degenerate graph $G$ with $\Delta(G)\le D$ satisfies that $\omega(G^2)=f(D)>\frac{5}{2}D$, then there is a subgraph $G^*$ of $G$ which is nice w.r.t.~a clique $S^*$ in 
${(G^*)}^2$ and $|S^*|\ge f(D)-60$.
\end{THMMAIN}

\begin{proof}
Suppose that $G$ is a 2-degenerate graph  with $\Delta(G)\le D$ and $S\subseteq V(G)$ is a clique in $G^2$ with $|S|=f(D)>\frac{5}{2}D$. Note that $G$ has a subgraph $G'$ which satisfies that
\begin{itemize}
\item[(1)]  $S$ is a clique of $G'^{2}$, and
\item[(2)] every edge $e\in E(G')$ has an endpoint in $S$.
\end{itemize}
Without loss of generality, we assume that $G$ and $S$ satisfy (1) and (2).  Let $\sigma$ be a vertex order witnessing that $G$ is 2-degenerate. Subject to this, choose $\sigma$ so the first vertex in $S$ appears as late as possible in $\sigma$.

Since $G$ is 2-degenerate with maximum degree at most $D$, the degeneracy of $G^2$ is  at most $(D-2)+2(D-1)=3D-4$. Therefore, $f(D)\le 3D-3$.  By using the discharging method, we show that there is a subgraph $G^*$ of $G$ which is nice w.r.t.~a clique $S^*$ in ${(G^*)}^2$ and $|S^*|\ge f(D)-60$.

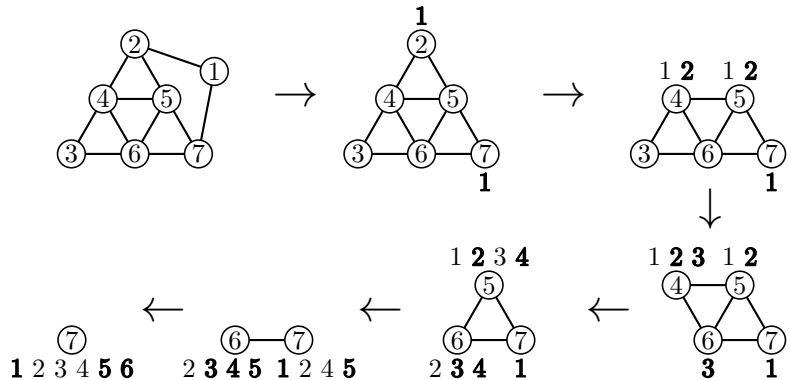
\begin{figure}[!h]
\centering
\begin{tikzpicture}[thick, scale=.75]
\tikzstyle{uStyle}=[shape = circle, minimum size = 10.5pt, inner sep = 0pt,
outer sep = 0pt, draw, fill=white, semithick]
\tikzstyle{lStyle}=[shape = circle, minimum size = 4.5pt, inner sep = 0pt,
outer sep = 0pt, draw, fill=none, draw=none]
\tikzset{every node/.style=uStyle}
\def\off{5mm}
\def\myRad{6.5mm}

\foreach \ang/\rad/\name in {90/2/2, 150/1/4, 30/1/5, 210/2/3, 270/1/6,
330/2/7, 30/2.5/1}
\draw (\ang:\rad*\myRad) node (v\name) {\footnotesize{\name}};

\foreach \x/\y in {1/2, 1/7, 2/4, 2/5, 3/4, 3/6, 4/5, 4/6, 5/6, 5/7, 6/7}
\draw (v\x) -- (v\y);

\begin{scope}[xshift=2in]
\foreach \ang/\rad/\name in {90/2/2, 150/1/4, 30/1/5, 210/2/3, 270/1/6,
330/2/7}
\draw (\ang:\rad*\myRad) node (v\name) {\footnotesize{\name}};

\foreach \x/\y in {2/4, 2/5, 3/4, 3/6, 4/5, 4/6, 5/6, 5/7, 6/7}
\draw (v\x) -- (v\y);

\foreach \tok/\x/\ang in {
{{\bolder 1}}/2/90,
{{\bolder 1}}/7/270
}
\draw (v\x) ++ (\ang:\off) node[lStyle] {\footnotesize{\tok}};
\end{scope}

\begin{scope}[xshift=4in]
\foreach \ang/\rad/\name in {150/1/4, 30/1/5, 210/2/3, 270/1/6,
330/2/7}
\draw (\ang:\rad*\myRad) node (v\name) {\footnotesize{\name}};

\foreach \x/\y in {3/4, 3/6, 4/5, 4/6, 5/6, 5/7, 6/7}
\draw (v\x) -- (v\y);

\foreach \tok/\x/\ang in {
{1 {\bolder 2}}/4/90,
{1 {\bolder 2}}/5/90,
{{\bolder 1}}/7/270
}
\draw (v\x) ++ (\ang:\off) node[lStyle] {\footnotesize{\tok}};
\end{scope}

\draw (10.2,-1.625) node[lStyle] {\Large{$\downarrow$}};

\begin{scope}[yshift=-.3in]
\begin{scope}[xshift=4in, yshift=-1in]
\foreach \ang/\rad/\name in {150/1/4, 30/1/5, 270/1/6,
330/2/7}
\draw (\ang:\rad*\myRad) node (v\name) {\footnotesize{\name}};

\foreach \x/\y in {4/5, 4/6, 5/6, 5/7, 6/7}
\draw (v\x) -- (v\y);

\foreach \tok/\x/\ang in {
{1 {\bolder {2 3}}}/4/90,
{1 {\bolder 2}}/5/90,
{{\bolder 3}}/6/270,
{{\bolder 1}}/7/270
}
\draw (v\x) ++ (\ang:\off) node[lStyle] {\footnotesize{\tok}};
\end{scope}

\begin{scope}[xshift=2.25in, yshift=-1in]
\foreach \ang/\rad/\name in {30/1/5, 270/1/6, 330/2/7}
\draw (\ang:\rad*\myRad) node (v\name) {\footnotesize{\name}};

\foreach \x/\y in {5/6, 5/7, 6/7}
\draw (v\x) -- (v\y);

\foreach \tok/\x/\ang in {
{1 {\bolder 2} 3 {\bolder 4}}/5/90,
{2 {\bolder{3 4}}}/6/270,
{{\bolder 1}}/7/270
}
\draw (v\x) ++ (\ang:\off) node[lStyle] {\footnotesize{\tok}};
\end{scope}

\begin{scope}[xshift=.7in, yshift=-1in]
\foreach \ang/\rad/\name in {270/1/6, 330/2/7}
\draw (\ang:\rad*\myRad) node (v\name) {\footnotesize{\name}};

\foreach \x/\y in {6/7}
\draw (v\x) -- (v\y);

\foreach \tok/\x/\ang in {
{2 {\bolder{3 4 5}}~~~~}/6/270,
{~~~~{\bolder 1} 2 4 \bolder 5}/7/270
}
\draw (v\x) ++ (\ang:\off) node[lStyle] {\footnotesize{\tok}};
\end{scope}

\begin{scope}[xshift=-.883in, yshift=-1in]
\foreach \ang/\rad/\name in {330/2/7}
\draw (\ang:\rad*\myRad) node (v\name) {\footnotesize{\name}};

\foreach \tok/\x/\ang in {
{{\bolder 1} 2 3 4 \bolder{5 6}}/7/270
}
\draw (v\x) ++ (\ang:\off) node[lStyle] {\footnotesize{\tok}};
\end{scope}

\foreach \x/\y in {0.5/-2.7, 4.35/-2.7, 8.4/-2.7}
\draw (\x,\y) node[lStyle] {\Large{$\leftarrow$}};

\end{scope}

\foreach \x/\y in {2.85/.34, 7.6/.34}
\draw (\x,\y) node[lStyle] {\Large{$\rightarrow$}};

\end{tikzpicture}
\caption{{\footnotesize A portion of a 2-degenerate graph $G$ and the positions of its tokens as
its vertices are deleted in a 2-degeneracy order; here
$S=\{1,\ldots,7\}$.  The number of each token indicates the vertex where
the token originated.  Primary tokens are shown in
bold. At each vertex, we show at most one token received from each other
vertex. This Figure is reproduced from [3] with permission of the authors.}} \label{fig:token}
\end{figure}

We delete the vertices of $G$ one by one according to their occurrences  in the order $\sigma$. Once a vertex $v$ in $S$ is deleted, vertex $v$ gives a ``primary" token to each of its neighbors later in $\sigma$. Moreover, once a vertex $v$ that currently possesses $s$ primary tokens (for some $s >1$) is deleted, $v$ gives to each of its neighbors later in $\sigma$ exactly $s$ ``secondary" tokens (see Figure~\ref{fig:token}). 
We emphasize here that Figure~\ref{fig:token} is reproduced from [3] with permission of the authors.  

\medskip
By this process, we show that there is a subgraph $G^*$ of $G$ which is 2-degenerate,  witnessed by a vertex order $\sigma^*$, and $G^*$ is nice w.r.t.~a clique $S^*$ in 
${(G^*)}^2$. Then, by discharging method, we show that  $|S|- |S^*|\le 60$. Let tokens($v$) (resp. primary($v$)) denote the number of tokens (resp. primary tokens) that each vertex $v$ holds immediately before it is deleted.

\begin{claim} If $v\in S$ and $w_1$, $w_2$ are the neighbors of $v$ later in $\sigma$, if they exist, then
\begin{equation}\label{equa:tokens}
1+\tv(v)+D +\pr(w_1)+\pr(w_2)+6\ge|S| >\frac{5}{2}D.
\tag{$\mathcal{F}$}
\end{equation}
\end{claim}
\begin{proof}
 Let $v\in S$. Recall that $S\subseteq N_{G^2}(v)$, i.e.,  $v$ is adjacent to every vertex in $S$ in the graph $G^2$. Let $w$ be a neighbor of $v$ in $G$ which precedes $v$ in $\sigma$. Then,  there is at most one vertex $x$ in $N_G(w)\cap S$ from which $v$ receives no token (neither primary nor secondary), so, the total number of such $x$ in $S$ is at most $D$. Each other adjacency in $G^2$ between $v$ and a verte in $S$ must be accounted for either (a) by a token received by $v$ or (b) by a primary token that has been or will be received by $w_1$ or $w_2$ or (c) by being an adjacency to $w_1$ or $w_2$ or to a neighbor of some $w_i$ that comes later than $w_i$ in $\sigma$  (as shown in Figure~\ref{fig:tokensv}  where each vertex in $\{v_1,v_2\}$, $\{v_3,v_4,v_5,v_8\}$ and $\{w_1,v_6,v_7,w_2,v_9,v_{10}\}$ are respectively accounted by (a), (b) and (c)).  
Note that at most $\tv(v)$ vertices of $S$ are handled by (a), and at most 6 vertices of $S$ are handled by (c). The number handled by (b) is at most $\pr(w_1) + \pr(w_2)$. This implies (\ref{equa:tokens}). Thus, the claim holds.
\end{proof}

\begin{figure}[htbp]
\begin{center}
\begin{tikzpicture}[u/.style={fill=black,minimum size =4pt,ellipse,inner sep=1pt},node distance=1.5cm,scale=0.9]
\node[u] (v00) at (-1,0){};
\node[u] (v0) at (0,0){};
\node[u] (v1) at (1,0){};
\node[u] (v2) at (2,0){};
\node[u] (v3) at (3,0){};
\node[u] (v4) at (4,0){};
\node[u] (v5) at (5,0){};
\node[u] (v6) at (6,0){};
\node[u] (v7) at (7,0){};
\node[u] (v8) at (8,0){};
\node[u] (v9) at (9,0){};
\node[u] (v10) at (10,0){};
\node[u] (v11) at (11,0){};
\node[u] (v12) at (12,0){};
\node[u] (v13) at (13,0){};

\draw (v00)  to [bend left=25] (v1);
\draw (v0)  to [bend left=35] (v4);
\draw (v1)  to [bend left=20] (v2);
\draw (v1)  to [bend left=30] (v4);
\draw (v3)  to [bend left=30] (v7);
\draw (v5)  to [bend left=18] (v7);
\draw (v7)  to [bend left=20] (v8);
\draw (v7)  to [bend left=30] (v9);
\draw (v4)  to [bend left=20] (v7);
\draw (v4)  to [bend left=30] (v11);
\draw (v6)  to [bend left=20] (v11);
\draw (v10)  to [bend left=18] (v11);
\draw (v11)  to [bend left=20] (v12);
\draw (v11)  to [bend left=30] (v13);
\draw[->, line width=1pt, >=latex, scale=2] (2.5,-0.5) -- (3.5,-0.5);

    \node[below=0.1cm,font=\small] at (v00) {$v_1$};
      \node[below=0.1cm,font=\small] at (v0) {$v_2$};
      \node[below=0.1cm,font=\small] at (v1) {$w$};
        \node[below=0.1cm,font=\small] at (v3) {$v_3$};
          \node[below=0.1cm,font=\small] at (v5) {$v_4$};
            \node[below=0.1cm,font=\small] at (v6) {$v_5$};
              \node[below=0.1cm,font=\small] at (v8) {$v_6$};  
              
                \node[below=0.1cm,font=\small] at (v9) {$v_7$};
                  \node[below=0.1cm,font=\small] at (v10) {$v_8$};
                  \node[below=0.1cm,font=\small] at (v12) {$v_9$};
                      \node[below=0.1cm,font=\small] at (v13) {$v_{10}$};
              
   \node[below=0.1cm,font=\small] at (v2) {$x$};
   \node[below=0.1cm,font=\small] at (v4) {$v$};
    \node[below=0.1cm,font=\small] at (v7) {$w_1$};
  \node[below=0.1cm,font=\small] at (v11) {$w_2$};
    \node[below=0.85cm,font=\small] at (v6) {later};
   \end{tikzpicture}
\end{center}
\caption{Adjacencies between $v$ and vertices in $S$ in $G$ w.r.t~occurrences in $\sigma$.}\label{fig:tokensv}
\end{figure}
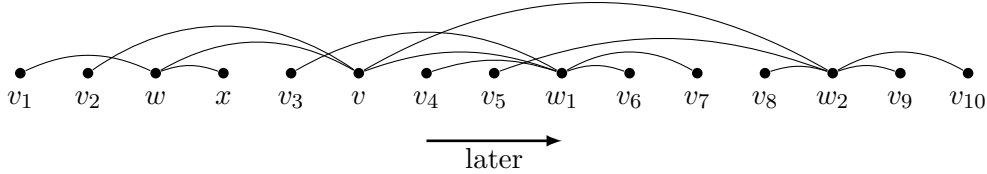

In order to analyze the movement of tokens, we define Big, Basic, and Nonbasic as follows.
\begin{definition}
\begin{enumerate}[(1)]
\item $\bi:= \{v \in V(G): \pr(v)>\frac{1}{4}D-4\}$.
\item $\ba:= \{v \in S: \tv(v)<\frac{1}{4}D-4\}$ and  $\no:=S\setminus \ba$.
\end{enumerate}
\end{definition}

\medskip
Note that $\bi\cap \ba=\emptyset$ while the relationship between  $\bi$ and $\no$ is unclear. Now, we consider the sum of the number of tokens at each vertex.  For each $v$ in $S$, $v$ gives primary tokens to at most two vertices, and each of those vertices gives secondary tokens (due to $v$) to at most two more vertices. Thus, the contribution of $v$ to the total number of tokens held by vertices at any point  is at most 6. So, 
\begin{equation}\label{tokens-upper}
\sum_{v\in V(G)}\tv(v)\le 6|S|\le 18D
\tag{$\mathcal{G}$}
\end{equation} 
which implies that $|\no\cup \bi|\le \frac{18D}{\frac{D}{4}-4}<73$.  

\medskip
We proceed to show the existence of $G^*,S^*$ and $\sigma^*$.

\begin{claim} \label{claim:vinbasic}
Each vertex $v \in \ba$ has two neighbors later in $\sigma$ that are both in $\bi$.
\end{claim}
\begin{proof}
Let $v \in \ba$ and let $w_1$, $w_2$ be the neighbors of $v$ later in $\sigma$, if they exist.  Combining  $v$ being basic with   inequaty~\ref{equa:tokens}, we have that  \[
\pr(w_1)+\pr(w_2)> \frac{5}{4}D-3.
\] 
Since each vertex of $G$ has at most $D$ neighbors  in $S$, we have that $\pr(w_i)\le D$ for each $i\in \{1,2\}$. This implies  that  $\pr(w_i)> \frac{1}{4}D-4$, for each $i\in \{1,2\}$. Thus, the claim holds.
\end{proof}

Let $W:= \{w : w \in N(v)$ for some $v \in\no$ and $w$ appears later in $\sigma$ than $v\}$.
Now we start the construction of  $G^*$, $ S^*$, $\sigma^*$.
\begin{construction}\label{con:CY}
\begin{itemize}
\item[(1)] Let $S^*:=\ba\setminus W$.
\item[(2)] Move $\bi$ to the end of $\sigma$ (after the final vertex of $S^*$) and move $(\no\cup W)\setminus\bi$
to the start of $\sigma$ (before the first vertex of $S^*$); call the resulting order $\sigma^*$.
\item[(3)] Delete every edge of $G$ with both endpoints outside $S^*$; call the resulting graph $G^*$.
\end{itemize}
\end{construction}
Now, we prove that $G^*$, $S^*$, $\sigma^*$  are exactly what we want. 

\begin{claim} $S^*$ appears consecutively in $\sigma^*$, and $S^*$ and $V(G^*)\setminus S^*$ are independent sets of $G^*$.
\end{claim}
\begin{proof}
First, we show that $S^*$ appears consecutively in $\sigma^*$. Suppose, 
to the contrary, that there is a vertex $v\notin S^*$ appears between $u\in S^*$ and $w\in S^*$ in the order $\sigma^*$.  Note that, by Construction~\ref{con:CY}(2), we have that $v\notin \bi\cup \no\cup W$. It implies that $v\notin S$ since $S=\ba\cup \no$. Thus, $v\in V(G)\setminus S$. Since the first vertex of $S$ appears as late as possible in $\sigma$ and each edge has an endpoint in $S$, we have that $N_{G}(v)\cap S\neq\emptyset$ which implies that $v$ is a later neighbor in $\sigma$ of some vertex $x$ in $S$, otherwise, we may move $v$ before the first vertex of $S$ which contradicts the choice of $\sigma$. If $x\in \ba$, then by Claim~\ref{claim:vinbasic}, we have that $v\in \bi$, a contradiction. Then, we have that $x\in \no$ which implies that $v\in W$ by the definition of $W$, a contradiction. Therefore, $S^*$ appears consecutively in $\sigma^*$. 

Observe that by Construction~\ref{con:CY}(3), $V(G^*)\setminus S^*$ is an independent set of $G^*$. Thus, it suffices to show that $S^*$ is an independent set of $G^*$.  Suppose that there is an edge between two vertices $u$ and  $w$  in $S^*$ where $w$ is a later neighbor of $u$ in $\sigma$. Since $u$ belongs to $\ba$, by Claim~\ref{claim:vinbasic}, we have $w\in \bi$, a contradiction. 
Therefore, the claim holds.
\end{proof}

The following two claims were proved in \cite{CY}. So, we just state the claims and those proofs.

\begin{claim}\label{claim:2-degen}
 $G^*$ is 2-degenerate, as witnessed by $\sigma^*$.
\end{claim}
\begin{proof}
 Consider steps (2) and (3) above. Suppose we move some $v\in \bi$ to the end of $\sigma$. If $v$
has a neighbor $w$ later (in $\sigma$), then $w\notin \ba$, so $w\notin S^*$. Thus, in step (3) we delete edge $vw$.
Suppose instead that we move a vertex $v\in W\cup \no \setminus \bi $ to the start of $\sigma$. If $v\in W\setminus\bi$ has a neighbor $w$ earlier (in $\sigma$), then by the definition of $W$, we have that $w\notin \ba$; if $v\in \no\setminus\bi$ has a neighbor $w$ earlier (in $\sigma$), then by Claim~\ref{claim:vinbasic}, we have that $w\notin \ba$. Thus, for both cases, we have that $w\notin S^*$ and so again in step (3) we delete $vw$. Therefore,  $\sigma^*$ is a 2-degeneracy order for $G^*$ because $\sigma$ is a 2-degeneracy order for $G$.
\end{proof}

\begin{claim}\label{claim:clique}
$S^*$
forms a clique in ${(G^*)}^2$.
\end{claim}

\begin{proof}
Every edge deleted when forming $G^*$ has both endpoints outside of $S^*$. But deleting such
edges cannot impact adjacency in ${(G^*)}^2$ between two vertices in $S^*$. Thus,  
${(G^*)}^2[S^*]=G^2[S^*]$.
\end{proof}

Now, we finish the proof of Theorem \ref{thm:modify} by showing that $|S\setminus S^*|\le 60$. 
\begin{claim}\label{claim:discharge}
$|S\setminus S^*|< 61$.
\end{claim}
\begin{proof}
Let $X:= \bi\cup \no\cup W$. Note that by Construction~\ref{con:CY}(1), we have that $|S\setminus S^*|\le |X|$.  To show that $|X|<61$, we use a discharging argument. Every $x$ in $X$ has an initial charge $\tv(x)$.  After distributing charge, we show that each $x$ in $X$ has charge at least $\frac{D}{4}-4$.  Then by inequality (\ref{tokens-upper}), we have \[
|S\setminus S^*|\le |X| \leq \frac{6|S|}{\frac{D}{4}-4} \leq \frac{18D}{\frac{D}{4}-4} \leq 73,
\]
since $D>1732$. So, by Theorem~\ref{thm:cynice}, we have $|S|\le \frac{5}{2}D+72$. 
Hence since $D>1732$, if we substitute the bound $|S|\le \frac{5}{2}D+72$ into the inequality $\sum_{v\in V(G)}\tv(v)\le 6|S|$ again, we have the following inequality
\[
|S\setminus S^*| \le |X|\le \frac{6|S|}{\frac{D}{4}-4}\le \frac{15D+6\times 72}{\frac{D}{4}-4}=60+\frac{12\times 144}{D-4}<61
\]
which completes the proof of the claim.

Now we distrbute the charges. For $x \in X$, let $w_1$ and $w_2$ be the neighbors of $x$ that follow $x$ in $\sigma$ (if they exist).
Note that  by the definition of $W$, it suffices to consider each $x\in \bi\cup \no$ (note that this union does not include $W$) and distribute charge so that 
$x$, $w_1$, and $w_2$ all finish with charge at least  $\frac{D}{4}-4$.

Suppose that $x\in \bi$. By definition, $\pr(x)>\frac{D}{4}-4$ which also implies that  $\tv(w_i)>\frac{D}{4}-4$  since $\tv(w_i)\ge \pr(x)$  for
each $i\in\{1, 2\}$. So $x$ does not need to give charge to any neighbor later in $\sigma$. Thus, $x$ (and each of its later neighbors) finishes with charge at least $\frac{D}{4}-4$.

Suppose that $x\in \no\setminus\bi$. If $\tv(x)\ge \frac{3}{4}D-12$, then $x$ gives charge $\frac{D}{4}-4$ to each of its at most two neighbors that follows it in $\sigma$. Thus, we assume that  $\frac{D}{2}-8\le \tv(x)<\frac{3}{4}D-12$. By inequality~(\ref{equa:tokens}),
we have that 
\[
\pr(w_1)+\pr(w_2)>\frac{3}{4}D+5
\] 
and so $\pr(w_i)\ge\frac{D}{4}-4$ for some $i\in\{1, 2\}$.
Thus, at least one later neighbor of $x$ is big. So $x$ gives charge $\frac{D}{4}-4$ to at most one neighbor later in $\sigma$. Thus, $x$ (and each of its later neighbors) finishes with charge at least $\frac{D}{4}-4$.

Finally, suppose that $\frac{D}{4}-4\le \tv(x)< \frac{D}{2}-8$. By inequality~(\ref{equa:tokens}), we have that 
\[\pr(w_1)+\pr(w_2)> D+1.\]
 Thus, $w_i\in \bi$ for at least one $i \in \{1, 2\}$, so $w_i$ needs no charge from $x$. If $w_1$, $w_2\in \bi$, then
neither $w_i$ receives charge from $x$, so $x$ (and each $w_i$) finishes with charge at least $\frac{D}{4}-4$. Suppose
instead, by symmetry, that $w_1\in \bi$ and $w_2\notin\bi$. Since $\pr(w_i)\le D$, we know that
\[
\tv(x)+\pr(w_2)>\frac{D}{2}-7.
\] 
Thus, $x$ gives to $w_2$ charge $\tv(x)-(\frac{D}{4}-4)$. Clearly, $x$ finishes with charge at least $\frac{D}{4}-4$. Furthermore, $w_2$ ends with charge at least 
\[\pr(w_2)+\tv(x)-\left(\frac{D}{4}-4 \right)>\frac{D}{2}-7-\left(\frac{D}{4}-4 \right)=\frac{D}{4}-3.
\]

Thus, each vertex of $X$ finishes with charge at least $\frac{D}{4}-4$.  \end{proof}

Therefore, Theorem \ref{thm:modify} holds.
\end{proof}

\end{document}